\documentclass[11pt]{amsart}
\usepackage{amssymb}
\usepackage{amsmath} 
\usepackage{amsthm} 
\usepackage{epsfig}
\usepackage{graphicx}
\usepackage{color}
\usepackage{transparent}
\usepackage{soul}
\usepackage{subfig}
\usepackage{wrapfig}
\usepackage[colorinlistoftodos]{todonotes}
 
\numberwithin{equation}{section}

\newtheorem*{thma}{Theorem A}

\newtheorem*{thmb}{Theorem B}

\newtheorem*{thmc}{Theorem C}

\newtheorem*{main}{Main Theorem}

\newtheorem{thm}{Theorem}[section]
\newtheorem{lem}[thm]{Lemma}
\newtheorem{defn}[thm]{Definition}
\newtheorem{cor}[thm]{Corollary}
\newtheorem{rem}[thm]{Remark}
\newtheorem{eg}[thm]{Example}
\newtheorem{prop}[thm]{Proposition}

\newcommand{\pic}{\text{Pic}}
\newcommand{\bptwo}{\text{Bir}(\mathbf{P}^2(\mathbb{C}))}
\newcommand{\ppc}{\mathbf{P}^2(\mathbb{C})}

\title{Salem/Pisot Numbers in the Weyl Spectrum} 

\author{Kyounghee Kim}
\address{Department of Mathematics\\
         Florida State University\\
         Tallahassee, FL 32308}
\email{kim@math.fsu.edu}


\begin{document}
\maketitle

\begin{abstract}
In this article, we define orbit data for birational maps of $\mathbf{P}^2(\mathbb{C})$ and show that this data uniquely determines the dynamical degree by providing minimal polynomials for dynamical degrees in terms of orbit data. Leveraging this relationship, we recursively identify all Salem and Pisot numbers that appear in the Weyl spectrum of the union of the Coxeter groups $W_n$ associated with $E_n$ and the set of all dynamical degrees of birational maps of $\mathbf{P}^2(\mathbb{C})$. Furthermore, we demonstrate that all accumulation points of Pisot numbers less than or equal to 2 are present in the Weyl spectrum.
\end{abstract}

\section{Introduction}\label{S:intro}

A Pisot number $\tau$ is an algebraic integer greater than $1$ whose Galois conjugates all have modulus $\lneq1$. A Salem number $\lambda$, defined by Salem \cite{Salem:1945}, is an algebraic integer whose minimal polynomial is of degree at least $4$, and all of whose Galois conjugates, except $\lambda$ and $\lambda^{-1}$, have modulus $1$. Salem and Pisot numbers are closely related in the sense that every Pisot number is a limit of a sequence of Salem numbers. However, it is not known that every limit of a sequence of distinct Salem numbers is a Pisot number. Salem and Pisot numbers have appeared in many different areas of Mathematics. (See \cite{Erdos:1939, Boyd:1977, Sury:1992, HironakaE:2001, Tanigawa:2004}.) For a detailed history, we refer to a comprehensive survey by Smyth \cite{Smyth:2015}.

\vspace{1ex}

Salem and Pisot numbers appear in Dynamical Systems. Diller and Favre \cite{Diller-Favre:2001} and Blanc and Cantat \cite{blancdynamical} showed that the exponential growth rate of the algebraic degree of a birational map on $\mathbf{P}^2(\mathbb{C})$ under iteration is given by either a Salem or Pisot number. This exponential growth rate of the degree is called \textit{the dynamical degree}. However, not every Salem/Pisot number can be realized. Blanc and Cantat \cite{blancdynamical} show that infinitely many Pisot numbers greater than the golden ratio cannot be realized as the dynamical degree of a birational map on $\mathbf{P}^2(\mathbb{C})$. 

\vspace{1ex}
According to Diller and Favre \cite{Diller-Favre:2001}, McMullen \cite{McMullen:2007}, Uehara \cite{Uehara:2010}, Kim \cite{Kim:2023}, and Blanc and Cantat \cite{blancdynamical}, the set of all dynamical degrees of birational maps on $\mathbf{P}^2(\mathbb{C})$ coincides with the Weyl spectrum $\Lambda_W$. This spectrum is defined as the closure of the union of all Salem numbers contained within the Coxeter groups $W_n$ associated with $E_n$, for $n \geq 3$. The primary objective of this article is to identify the Salem and Pisot numbers within the Weyl spectrum and investigate their properties as dynamical degrees of birational maps on $\mathbf{P}^2(\mathbb{C})$.

\vspace{1ex}
Considering a family of Jonqui\'{e}res maps of degree $d\ge 2$, Bot \cite{Bot:2024} studied the ordinal of the dynamical degrees of all birational maps on $ \mathbf{P}^2(\mathbb{C})$. Noice that there are birational maps that are not conjugate to Jonqui\'{e}res maps. For instance, the Salem number $\delta \approx 1.96683$, the largest root of \[ t^{12}-2 t^{11} + t^{10} - 2 t^9+t^7-t^6+t^5-2 t^3+t^2-2 t+1 \] is the spectral radius of the element in Wyle Spectrum. (See Example \ref{E:noJ}.) However, this Salem number is not realized as a dynamical degree of a Jonqui\'{e}res map.

\vspace{1ex}
 In our paper, we consider the set of all dynamical degrees of birational maps on $\mathbf{P}^2(\mathbb{C})$. 

\vspace{1ex}

In this work, we leverage theories from Coxeter groups and birational maps on $\mathbf{P}^2(\mathbb{C})$ to show that the Weyl spectrum can be fully determined recursively with respect to the Weyl degree.

\begin{main} The Weyl spectrum $\Lambda_W$ is determined recursively with respect to the Weyl degree. \end{main}

For each element $\omega \in \bigcup_n W_n$, we define a Weyl degree, $\deg(\omega) \in \mathbb{N}$ (see Section \ref{SS:wdeg} for a precise definition). Let $\Omega_d$ represent the set of all elements in $\bigcup_n W_n$ with Weyl degree equal to $d$, so that $\bigcup_n W_n = \bigcup_d \Omega_d$. We will identify all Salem numbers in $\Lambda_W$ by recursively identifying the Salem numbers in each $\Omega_d$ for all $d$.

\vspace{1ex}

The Main Theorem is proved by establishing the following technical result:

%
%
%
%
%

\begin{thma}
The Weyl spectrum $\Lambda_W$ is given by 
\[ \Lambda_W = \bigcup_{d\ge2} \Lambda_d,\] where $\Lambda_d$ is determined recursively with respect to $d$: 
\[ \Lambda_d =  \{ \rho(\chi_{M, \bar n}): M \in (\mathcal{M}_d \cup  \mathcal{M}'_d), \bar n \in \mathbb{N}^k, k = |M|-1\}. \]
Here, $\rho$ denotes the largest real root, $|M|$ denotes the size of $M$, $\chi$ is defined in (\ref{E:charpoly}), $\mathcal{M}_d$ for $d \ge 2$ is a finite set defined recursively by Theorem \ref{T:allbase}, and $\mathcal{M}'_d$ for $d \ge 2$ is defined in (\ref{E:setM}).
\end{thma}

\vspace{1ex}
To prove this theorem, we show in Section \ref{S:coxeter} that the set of all Salem numbers in $\Omega_d$ is determined by a finite set of matrices, $\mathcal{M}_d$, along with ordered lists of positive integers. We also provide a formula (\ref{E:charpoly}) for the characteristic polynomial (up to a product of cyclotomic factors) of each element in $\Omega_d$. In Section \ref{S:sp}, we establish the recursive formula for $\mathcal{M}_d$ and determine $\mathcal{M}'_d$ for Pisot numbers.

\begin{rem}
According to Amara \cite[Theorem~3.1]{Amara:1966}, the set of all accumulation points of Pisot numbers less than or equal to 2 consists precisely of the algebraic numbers \[ \alpha_1 =\beta_1<\alpha_2<\beta_2<\alpha_3<\hat \theta'_2 < \beta_3<\alpha_4 <\cdots <\alpha_n<\beta_n <\cdots <2 \]  where $\alpha_n$ is the largest real root of $t^{n+1} - 2t^n + t - 1 $, $\beta_n$ is the largest real root of $t^{n+1} - (t^{n+1} - 1)/(t - 1) $, and $\hat{\theta}'_2$ is the largest real root of $t^4 - t^3 - 2t^2 + 1 $. 

\vspace{1ex}
To study the distribution of Salem numbers in the Weyl spectrum, Theorem A is particularly useful, as it provides a recursive formula in terms of the Weyl degree. For instance, in Section \ref{ApendA}, we demonstrate that all limit points of Pisot numbers less than or equal to 2 are represented in the Weyl spectrum. Furthermore, we construct sequences of Pisot numbers and Salem numbers within the Weyl spectrum that approach each limit point from below, thereby providing an explicit, infinite list of Pisot numbers between any two successive accumulation points.

 
 \end{rem}

\begin{rem}
Let $P \approx 1.6217$ be the largest real root of the polynomial \[ t^{16}-2 t^{15}+2 t^{14}-3 t^{13}+ 2 t^{12}-2 t^{11} + t^{10} + t^7-t^6+2 t^5-2 t^4+2 t^3-2 t^2+t-1.\]
Boyd first referenced this Pisot number in \cite{Boyd:1977}. Experimental computations suggest that the spectral radius of an element of Weyl degree $d \le 5$ is either less than the golden ratio or greater than $1.624$. Although the spectral radius is not necessarily monotonic with respect to the Weyl degree, we conjecture, based on computational evidence, that $P$ is absent from the Weyl spectrum. 
\end{rem}

For a birational map $f \in \bptwo$, there is the set $\mathcal{E}(f)$ of exceptional curves mapped to a point and the set $\mathcal{I}(f)$ of points of indeterminacy. Any irreducible curve in $\mathcal{E}(f)$ maps to a point under $f$.  Let us define $\mathcal{E}^*(f)$ as the set of exceptional curves mapped back to a curve:
\[ \mathcal{E}^*(f) = \{ C \in \mathcal{E}(f) : f^{n} (C) \in \mathcal{I}(f), \ \ \dim f^{n+1}(C) =1, n\ge 1 \} \]
For each $C_i \in \mathcal{E}^*(f)$, let $n_i$ be the smallest number of iterations $n$ such that $\dim f^{n+1}(C_i) =1$. We call the ordered tuple $(n_1, \dots, n_k)$, where $k=|\mathcal{E}^*(f)|$, the orbit lengths of $f$. 
In Section \ref{S:birational}, we define \textit{the orbit data} of $f$ as the numerical data obtained by orbits of curves in $\mathcal{E}^*(f)$. We show that the dynamical degree increases with respect to the orbit lengths. 
\begin{thmb}
Let $f$ and $g$ are birational maps on $\mathbf{P}^2(\mathbb{C})$. Suppose the orbit data of $f$ and $g$ are distinct, given by $\mathcal{O}(f)=(M, (n_1, \dots, n_k))$ and $\mathcal{O}(g)=(M, (n'_1, \dots, n'_k))$ respectively. Assume the following conditions hold:
\begin{enumerate}
\item  $n_i \le n'_i$ for all $i=1, \dots, k$,
\item $n_i \ne n'_i$ for some $i$, 
\item $\lambda(f)>1$.
\end{enumerate}
Then $\lambda(g) \gneq \lambda(f)$. 
\end{thmb}

The smallest accumulation point of Pisot numbers is the golden ratio. With the construction provided by Diller \cite{Diller:2011}, we show: 
\begin{thmc}
All Pisot numbers smaller than or equal to the golden ratio are realized as dynamical degrees of birational maps on $\mathbf{P}^2(\mathbb{C})$. 
\end{thmc}

\vspace{2ex}
 
 The Coxeter group $W_n, n\ge 4$, associated with $E_n$ diagram in Figure \ref{fig:coxetergraph} is the group with generators $s_i$'s and relations \[ (s_i s_j)^{m_{i,j}} = 1 \qquad \text{where}\ \ m_{i,j} \ =\   \left\{ \ \begin{aligned} & 1 \ \ \text{ if } i=j,\\ &  2 \ \ \text{ if there is no edge joining } s_i \text{ and } s_j, \\ & 3 \ \ \text{ if there is an edge joining } s_i \text{ and } s_j. \\ \end{aligned} \right. \]

 \begin{figure}[h]
\centering
\includegraphics{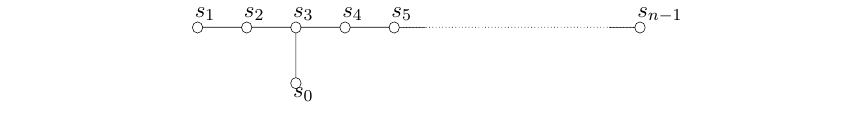}
\caption{Coxeter Graph $E_n$ \label{fig:coxetergraph} }
\end{figure}
For $\omega \in W_n$, we say $\omega$ is \textit{essential} if, for all $\rho \in W_n$, the conjugate $\rho \omega \rho^{-1}$ is not an element of a proper subgroup generated by $s_i$, where $ i \in I \subsetneq \{0, 1, \dots, n-1\}$. The Coxeter group $W_n$ has a faithful reflection representation on a vector space $V_n$ with the inner product $B_n(\alpha_i,\alpha_j) = -2 \cos (\pi/m_{ij})$ on the basis $\{\alpha_i\}$ dual to $\{s_i\}$. Let $\lambda(\omega)$ denote the leading eigenvalue of $\omega|_{V_n}$. McMullen \cite{McMullen:2002} showed that if $\omega \in W_n, n\ge 10$ is essential, then the spectral radius $\lambda(\omega)$ is always greater than or equal to $\lambda_{10}$. The Salem number $\lambda_{10} \approx 1.176$, known as Lehmer's number, has the minimal polynomial:   \[ L(t) = t^{10} +t^9-t^7-t^6-t^5-t^4-t^3+t+1. \]

Let us define $\Lambda_{S,W}$ as the set of all spectral radii of essential elements in $W_n, n\ge 10$: \[ \Lambda_{S,W} := \{ \lambda(\omega): \omega \text{ is an essential element in } W_n, n\ge 10 \}. \]  Since every element in $W_n, n\le 9$ has spectral radius equal to $1$, we have \[  \{ \lambda(\omega): \omega\in W_n, n\ge 4 \} = \{1\} \cup \Lambda_{S,W}.\]
We denote the closure of all spectral radii of $W_n, n\ge 4$ as the Weyl spectrum, $ \Lambda_W$: \[ \Lambda_W \ =\  \{1\} \cup \overline{\Lambda_{S,W}}.\]

 It is known \cite{Lakatos:2001, Lakatos:2010} that the spectral radius of $\omega \in W_n$ is either $1$, a reciprocal quadratic integer, or a Salem number. A reciprocal quadratic integer $q_n$ is an algebraic integer whose minimal polynomial is $1- n t + t^2$ for some integer $n$. By Blanc and Cantat \cite{blancdynamical}, every positive reciprocal quadratic integer is realized as a spectral radius of an element of $W_n$. The smallest positive reciprocal quadratic integer greater than $1$ is $q_3 \approx 2.618$. This quadratic integer, $q_3$, is realized as a spectral radius of an element $\omega_q\in W_{10}$ where 
\[  \omega_q =s_1   s_7   s_9   \kappa_{2,3,4}   \kappa_{2,5,6}   \kappa_{2,8,10}   \kappa_{3,4,5}   \kappa_{6,8,10}, \]
\[
  \kappa_{i,j,k} = \rho_{1,i}   \rho_{2,j}  \rho_{3,k}   s_0    \rho_{1,i}   \rho_{2,j}  \rho_{3,k}, \qquad\text{and} \qquad \rho_{i,j} = s_i s_{i+1} \cdots s_{j-1}. \]
On the other hand, not every Salem number is realized as a spectral radius of an element of $W_n$. However, McMullen \cite{McMullen:2002} showed that the smallest known Salem number, $\lambda_{10}$, is indeed realized as a spectral radius of the Coxeter element in $W_{10}$. 
Thus, $\lambda_{10}$ is the smallest number greater than $1$ in $\Lambda_W$.
 
 \vspace{1ex}
 As noted above, any limit of Salem numbers is either a Salem or Pisot number. Consequently, the Weyl spectrum consists of $1$, Salem numbers, and Pisot numbers. Let us denote by $\Lambda_{P,W}$ the set of limit points of $\Lambda_{S,W}$:
 \[ \Lambda_W = \{1\} \cup \Lambda_{S,W} \cup \Lambda_{P,W}. \]
 
 \vspace{1ex}
 On the other hand, the Weyl spectrum $\Lambda_W$ is realized as the set of dynamical degrees of birational maps on $\mathbf{P}^2(\mathbb{C})$. Consider a birational transformation $f=[f_1:f_2:f_3]$ on $\mathbf{P}^2(\mathbb{C})$, where $f_1, f_2$ and $f_3$ are polynomials of degree $d$ without a non-constant common factor. The degree of $f$ is given by the common degree $d$ of the coordinate polynomials, and the dynamical degree $\lambda(f)$ is defined by \[ \lambda(f) \ =\ \lim_{n \to \infty} (\deg f^n)^{1/n}. \]
 Because of possible common factors in the $n$-fold composition of $f$, we have $\deg (f^n) \le (\deg f)^n$, and thus, computing the dynamical degree of a given birational map is a challenging task. For a given birational map $f \in \bptwo$, Diller and Favre \cite{Diller-Favre:2001} showed that there is always a rational surface $X$, obtained by blowing up $\mathbf{P}^2(\mathbb{C})$ along a finite set of points, such that the induced map $F:X \to X$ satisfies \[ \lambda(F) = \lambda(f) = \rho(F_*) \] where $\rho$ denotes the spectral radius and $F_*$ is the induced linear map on $\pic(X)$. 
 
The dynamical degree of an automorphism on $\mathbf{P}^2(\mathbb{C})$ is equal to $1$. For a birational map with a dynamical degree $\lambda(f) > 1$, Diller and Favre \cite{Diller-Favre:2001} and Blanc and Cantat \cite{blancdynamical} provide the following classification:
 \begin{thm}\cite{Diller-Favre:2001, blancdynamical}\label{inT:ddeg} Suppose $f \in \bptwo$ is a birational transformation with a dynamical degree $\lambda(f) >1$. Then we have one of the following:
 \begin{itemize}
 \item $f$ is birationally equivalent to an automorphism of the rational surface, and the dynamical degree $\lambda(f)$ is either a Salem number or a reciprocal quadratic integer.
 \item $f$ is not birationally equivalent to an automorphism, and the dynamical degree $\lambda(f)$ is a Pisot number.   
 \end{itemize}
 \end{thm}
 
Let us denote by $\Lambda_\mathbb{C}$ the set of all dynamical degrees of birational maps on $\mathbf{P}^2(\mathbb{C})$. 
Let us also denote by $\Lambda_{S, \mathbb{C}} \subset \Lambda_\mathbb{C}$ the set of all dynamical degrees of birational maps equivalent to an automorphism, and by $\Lambda_{P, \mathbb{C}} \subset \Lambda_\mathbb{C}$ the set of all dynamical degrees of birational maps that are not equivalent to an automorphism. Then, by Theorem \ref{inT:ddeg}, the set $\Lambda_{S, \mathbb{C}}$ consists of $1$, reciprocal quadratic integers, and Salem numbers, and the set $\Lambda_{P, \mathbb{C}}$ consists of Pisot numbers. 
By Nagata \cite{Nagata, Nagata2}, if $F$ is an automorphism on a rational surface $X$ over $\mathbb{C}$ with infinite order, then the induced action $F_*$ on $\text{Pic}(X)$ is an element of $W_n$, where $n+1$ is the Picard rank of $X$. Thus we have $\Lambda_{S, \mathbb{C}} \subset \Lambda_{S, W}$. McMullen \cite{McMullen:2007}, Uehara \cite{Uehara:2010}, and Kim \cite{Kim:2023} show that, in fact,
\[ \Lambda_{S, \mathbb{C}} = \Lambda_{S, W}. \] 
Also, Blanc and Cantat \cite{blancdynamical} show that the dynamical spectrum $\Lambda_\mathbb{C}$ is well-ordered and closed. It follows that \[  \Lambda_W = \Lambda_\mathbb{C}.\] To simplify our notations, let us use $\Lambda_{S}$ for both $\Lambda_{S, W}$ and $\Lambda_{S, \mathbb{C}}$. 

\begin{rem}
Since $\Lambda_W = \Lambda_\mathbb{C}$, the closure of Salem numbers in the Weyl spectrum consists of both Pisot and Salem numbers. Thus, in principle, there could be a Salem number in $\Lambda_{P,W} \setminus \Lambda_{S,W}$. However, by combining the identities $\Lambda_{S, \mathbb{C}} = \Lambda_{S, W}$ and $\Lambda_W = \Lambda_\mathbb{C}$ with the result of Blanc and Cantat \cite{blancdynamical}, we conclude that if a Salem number $\lambda$ is in $\Lambda_{P,W}$, then $\lambda$ must also lie in $\Lambda_S$. Consequently, there is no Salem number in $\Lambda_{P,W} \setminus \Lambda_{S,W}$.
\end{rem}

This article is organized as follows: In Section \ref{S:birational}, we discuss basic facts about birational maps on $\mathbf{P}^2(\mathbb{C})$ and define orbit data for these maps. In Section \ref{S:coxeter}, we explore the properties of the Coxeter group and give a formula for the characteristic polynomial of any element within the group. The proof of Theorem A is in Section \ref{S:sp}, and the proof of Theorem B is in Section \ref{S:spec}, and Section \ref{S:pisot} lists orbit data corresponding to Pisot numbers smaller than or equal to the golden ratio, including the proof of Theorem C. Section \ref{ApendA} addresses accumulation points of Pisot numbers in the Weyl spectrum.

\section{Birational maps on $\mathbf{P}^2(\mathbb{C})$}\label{S:birational}
Suppose $X$ is a rational surface over $\mathbb{C}$. Let $f: X \dasharrow X$ be a birational transformation on $X$. 
Then, there is a well-defined linear map $f_* : \pic(X) \to \pic(X)$, and \textit{the dynamical degree} $\lambda(f)$ of $f$ is defined by 
\[ \lambda(f) = \lim_{n \to \infty} \rho((f^n)_*) ^{1/n}, \] where $\rho$ denotes the spectral radius. 

\vspace{1ex}
We say a birational transformation $f$ on a rational surface $X$ is \textit{algebraically stable} if 
\begin{equation}\label{E:Astable} (f^n)_* = (f_*)^n \qquad \text{for all } n \in \mathbb{Z}.  \end{equation} Thus, if a birational transformation $f$ is algebraically stable, then the dynamical degree $ \lambda(f)$ of $f$ is given by the spectral radius of the induced linear action $f_*$ on $\pic(X)$.

\subsection{$ \text{Bir}(\mathbf{P}^2(\mathbb{C}))$} 
Let $f$ be a birational map on $\mathbf{P}^2(\mathbb{C})$. There are three homogeneous polynomials $f_1,f_2$ and $f_3$ of the same degree $d$ and without a non-constant common factor such that 
\[ f[x_1:x_2:x_3] \ = \ [ f_1(x_1,x_2,x_3):f_2(x_1,x_2,x_3):f_3(x_1,x_2,x_3)]. \]
In this case, \textit{the degree of $f$} is given by the common degree of $f_i$'s: \[ \deg (f) = d.\] 

Because of possible common factors in the iterations of $f$,  the degree of $f^2= f \circ f$ is not necessarily equal to $\deg(f)^2$. For example, for the \textit{Cremona involution}  \[ \sigma[x_1:x_2:x_3] = [x_2 x_3:x_1x_3:x_1 x_2] ,\]
we have $\deg(\sigma) = 2$ and $\deg(\sigma^2) = 1$. In fact, we have \[ \deg (\sigma^n) \lneq (\deg \sigma)^n\ \ \ \text{ for all } n \ge 2.\]

Since $\pic(\mathbf{P}^2(\mathbb{C}))$ is generated by a generic line, the dynamical degree $\lambda(f)$ is also given by \[ \lambda(f) = \lim_{n \to \infty} (\deg f^n)^{1/n}. \]

Determining $\deg f^n$ for all $n$ is very challenging due to possible cancellations in the coordinate functions. However, Diller and Favre \cite{Diller-Favre:2001} showed that one can always construct a rational surface $X$ by blowing up finitely many (possibly infinitely near) points on $\mathbf{P}^2(\mathbb{C})$ such that the induced map $F$ on $X$ is algebraically stable. 
Since the dynamical degree is invariant under birational conjugacy, we have $\lambda(f) = \rho(F_*)$.

\subsection{$n$-fold Composition}
For any curve $V \subset \mathbf{P}^2(\mathbb{C})$, let us denote $f(V) = \overline{f(V- \mathcal{I}(f))}$ as the strict transform of $V$ under $f$.
Let us denote $f^n(V)$ as the image of $V$ under the $n$-fold composition of $f$. Since the indeterminacy locus $\mathcal{I}(f^n)$ consists of a finite set of points, $f^n(V)$ is well-defined for any curve $V \subset \mathbf{P}^2(\mathbb{C})$. However, it may not be possible to iterate $V$ under $f$ $n$-times; that is, it is possible for $f^j(V)$ to lie in $\mathcal{I}(f)$ for some $j=1, \dots, n-1$. 

\vspace{1ex}

For example, consider the Cremona involution $\sigma$ where $\sigma^2=Id$:
\[ \sigma^2( \{ x_1 =0 \} ) = \{x_1 =0\}, \qquad \text{but} \qquad \sigma( \sigma(\{x_1=0\})) = \sigma([1:0:0]) \text{ is not defined} .\]

\subsection{Orbit Data}
Suppose $f \in \bptwo$ is given by homogeneous polynomials $f_1,f_2$ and $f_3$ of the same degree $d$ without a non-constant common factor.
The exceptional locus $\mathcal{E}(f)$ is given by the zero set of the determinant of the Jacobian matrix $J_f$. 

Suppose $J_f = \prod_{i=1}^m \phi_i$, a product of (not necessarily distinct) irreducible polynomials, and let $c_i = \deg \phi_i$. Also let $C_i = Z(\phi_i) \subset \mathbf{P}^2(\mathbb{C})$ be the zero set of $\phi_i$. Then $C_1, \dots, C_m$ are the degree $c_i$ irreducible components of $\mathcal{E}(f)$, counted with multiplicity of the defining polynomials in $\det J_f$: \[ \mathcal{E}(f) = \{ C_1, \dots, C_m \}, \quad \text{with} \ \ \deg C_i = c_i . \] 

Since the degree of $f$ is equal to $d$, the determinant $\det J_f$ is given by a polynomial of degree $3(d-1)$. By counting the degree, we see that $\sum c_i = 3 (d-1)$. 

\vspace{1ex}
In the example \ref{Eg:tau}, we see that a birational map $\tau[x_1:x_2:x_3] = [x_1^2: x_1 x_2: x_2^2-x_1 x_3]$ satisfies $\det J_\tau=-2 x_1^3$. Thus $\mathcal{E}(\tau) = \{ \{ x_1=0\}, \{x_1=0\},\{x_1=0\} \}$ consists of three identical lines.

Also, there is an indeterminacy locus: \[ \mathcal{I}(f) = \{ z\in \mathbf{P}^2(\mathbb{C}): f_1(z)=f_2(z) = f_3(z) =0 \}. \]

\vspace{1ex}

\subsubsection{Orbits of Exceptional curves} Let us define a subset $\mathcal{E}^*(f) \subset \mathcal{E}(f)$ by
\[ \mathcal{E}^*(f) = \{ C_i \in \mathcal{E}(f) : \text{dim} f^{n+1}(C_i) =1 \ \ \text{ for some } n \ge 1\}. \] It follows that if $C_i \in \mathcal{E}^*(f)$ then  $f^j(C_i) \in \mathcal{I}(f)$ for some $j \ge 1$. 

\vspace{1ex}

With renaming of the $C_i$'s if necessary, we may assume that $\mathcal{E}^*(f) = \{ C_1, \dots, C_k \}$ for some $k\le m$. For $i=1, \dots, k$, let $n_i$ be the smallest positive integer $n$ such that $\text{dim} f^{n+1}(C_i) =1$:
\[ n_i = \min \{ n\ge 1: \text{dim} f^{n+1}(C_i) =1\}.\]
Also let $\mathcal{O}_i$ be the set defined by:
\[ \mathcal{O}_i = \{ p_{i,j}=f^j (C_i), 1 \le j \le n_i \}. \]
Let us call the set $\mathcal{O}_i$ \textit{the orbit} of $C_i$ under $f$. Note that this is not the typical definition of orbit, as $f^j(C_i)$ could lie in the indeterminacy locus of $f$. If $\mathcal{O}_i$ is the orbit of $C_i \in \mathcal{E}^*(f)$, then we have: \[\{ p_{i,1}, 1 \le i \le k \} \subset \mathcal{I}(f^{-1})\quad \text{and} \quad \{ p_{i,n_i}, 1\le i \le k \} \subset \mathcal{I}(f). \]

\vspace{1ex}
Constructing an algebraically stable modification is a challenging task. However, condition (\ref{E:Astable}) provides a criterion for determining which exceptional curves require modification.

\begin{prop}\label{P:Astable}
Suppose $f$ is a birational map with an exceptional set $\mathcal{E}(f) = { C_i \mid i = 1, \ldots, m }$ consisting of irreducible (not necessarily distinct) exceptional curves. An algebraically stable modification $F:X \dasharrow X$ is obtained by blowing up the orbit of $C_i$ for $i \in I \subset \{1, 2, \dots, m\}$ if and only if, for each $i \in I$, there exists a positive integer $n_i$ such that $\dim f^{n_i + 1}(C_i) = 1$, and for each $i \notin I$, we have $\dim f^j(C_i) = 0$ for all $j \geq 1$.
\end{prop}

\begin{proof}
Suppose an algebraically stable modification $F: X \dasharrow X$ is obtained by blowing up the orbit of $C_i$ for $i \in I \subset {1, 2, \dots, m}$. Then, for each $i \in I$, the proper transform $F_*[C_i]$ is an element in $\pic(X)$ with self-intersection $-1$. Hence, for all $n \geq 1$, $(F_*)^n[C_i]$ cannot vanish. Since $X$ is obtained by blowing up finitely many points, there exists a positive integer $n_i$ such that $(F_*)^{n_i + 1}[C_i] \cdot e_0 \neq 0$, where $e_0$ is the class of a generic line. If $i \notin I$, then $F_*[C_i] = 0$, and thus $(F_*)^n[C_i] = 0$ for all $n \geq 1$.

Conversely, suppose $C_i \in \mathcal{E}(f)$ satisfies $\dim f^{n_i + 1}(C_i) = 1$ for some positive integer $n_i$, and assume that the orbit of $C_i$ has not been blown up in the construction of $X$. Then $F_*[C_i] = 0$. However, the condition $\dim f^{n_i + 1}(C_i) = 1$ implies that $(F_*)^{n_i + 1}[C_i] \cdot e_0 \neq 0$. This contradicts the assumption that $(F_*)^{n_i + 1} = (F^{n_i + 1})_*$, completing the proof.
\end{proof}

\vspace{1ex}

\subsubsection{Algebraically stable modification}\label{SSS:as} According to Diller and Favre \cite{Diller-Favre:2001}, for a birational map $f$ on $\ppc$, there is a modification $X$ such that the lift of $f$ is algebraically stable. Let us summarize the construction of the modification in our setting:

\vspace{1ex}
Without loss of generality, we may assume $|\mathcal{O}_i| \le |\mathcal{O}_{i+1}|$ for all $i=1, \dots, k$.
Let us define a sequence of blowups 
\[ \pi: X= X_k \to X_{k-1} \to \cdots \to X_1 \to X_0= \mathbf{P}^2(\mathbb{C}),\] where:
\begin{enumerate}\addtolength{\itemsep}{1ex}
\item Let $\tilde{D_1}=f^{n_1+1} (C_1) \subset \ppc$ and let $\tilde{p}_{1, j} = p_{1,j}$ for $j=1, \dots, n_1$.
 \item $\pi_1: X_1 \to X_0 = \mathbf{P}^2(\mathbb{C})$ is a blowup along the finite set of points in $\mathcal{O}_1$. 
\item Let $f_1: X_1 \dasharrow X_1$ be the lift of $f$. Also let $\tilde{C_2}\subset X_1$ be the strict transform of $C_2$ and let $\tilde{D_2} = f_1^{n_2+1} (\tilde C_2)$. Then $\tilde{D_2}$ is a curve in $X_1$ and the orbit of $\tilde{C_2}$ under $f_1$ is given by  
\[ \tilde{ \mathcal{O}_2} = \{ \tilde{p}_{2, j} = f_1^j{\tilde{C}_2}, j=1 \le j \le n_2\} \subset X_1.\]
Notice that the orbit $\tilde{ \mathcal{O}_2}$ is the strict transform of $\mathcal{O}_2 =\{ p_{2,j}, 1 \le j \le n_2\} \subset X_1$.
 \item For $i=2, \dots, k-1$, $\pi_i: X_i \to X_{i-1}$ is a blowup along the finite set of points in $\tilde{\mathcal{O}_i}$.  
 Then, let $f_i: X_i \dasharrow X_i$ be the lift of $f_{i-1}$. Let $\tilde{C}_{i+1}\subset X_i$ be the strict transform of $C_{i+1}$ and let $\tilde{D}_{i+1} = f_i^{n_i+1} (\tilde C_{i+1})$, and let $\tilde{\mathcal{O}}_{i+1}$ be the orbit of $\tilde{C}_{i+1}$ under $f_i$:
 \[ \tilde{ \mathcal{O}}_{i+1} = \{ \tilde{p}_{i+1, j} = f_i^j{\tilde{C}_{i+1}}, j=1 \le j \le n_{i+1}\} \subset X_i.\]
\item  $\pi_k: X_k \to X_{k-1}$ is a blowup along the finite set of points in $\tilde{\mathcal{O}}_k$.
\item Let $\pi= \pi_1 \circ \pi_{2} \circ \cdots \circ \pi_k$. Then the lift $F:X \dasharrow X$ of $f$ is algebraically stable. 
\end{enumerate}

We refer to the proof of Theorem 0.1 in \cite{Diller-Favre:2001} for detailed information about this construction and to \cite{Bedford-Kim:2010, Bedford-Kim:rotation} for more examples. If the orbits of exceptional curves are not disjoint, the construction becomes more complex. This complexity is illustrated in Example \ref{Eg:tau}.

\vspace{1ex}
\subsubsection{Induced Action on the Picard group} With the above construction, we have \[ \pic(X) = \langle e_0, e_{i,j}, 1 \le j\le n_i, 1 \le i \le k \rangle,\] where $e_0$ is the class of a generic line and $e_{i,j}$ is the class of the total transform of $\tilde{p}_{i,j}$. Recall that the point $\tilde{p}_{i,j} \in X_{i-1}$ is blown up by $\pi_i$ for the first time. Thus, $e_{i,j}$ is the class of $\pi_k^*\circ \pi_{k-1}^* \circ \cdots \circ \pi_i^* (\tilde{p}_{i,j})$. This set of generators provides a geometric basis for $\pic(X)$: 
  \[ e_0 \cdot e_0 = 1, \quad e_{i,j} \cdot e_{i,j} = -1 \ \ \text{for } 1 \le j\le n_i, 1 \le i \le k, \qquad \text{ and} \qquad   e_s \cdot e_t =0 \ \ \text{if } s \ne t\]
 It is convenient to work with a geometric basis in the following sense:
 
 Suppose $\hat C_i, \hat D_i \subset X$ are the total transform of $\tilde C_i, \tilde D_i $ (respectively) in $X_{i-1}$ for $i=1, \dots, k$. Then, the induced action of $F_*$ on $\pic(X)$ satisfies:
 \[ F_* : [ \hat C_i] \mapsto e_{i,1} \mapsto e_{i,2}\mapsto \cdots \mapsto e_{i,n_i} \mapsto [\hat D_i]\]

\vspace{1ex}
The class $[\hat D_i] \in \pic(X)$ is given by a linear combination of $e_0$ and $e_{i,1}$'s
\begin{equation}\label{E:fstar1} [\hat D_i] = F_*(e_{i,n_i}) = d_i e_0 - \sum_j m_{j,i} \,e_{j,1}, \qquad \text{where } \ \ d_i >0 \text{, and } m_{j,i} \ge 0.\end{equation} Since a generic line in $\mathbf{P}^2(\mathbb{C})$ intersects $C_i$ at $c_i$ points (counted with multiplicity) with $c_i>0$, we have 
\begin{equation}\label{E:fstar2} F_*(e_0) = d e_0 - \sum_i c_i \, e_{i,1}\end{equation} and\begin{equation}\label{E:fstar3} F_*(e_{i,j}) = e_{i,j+1}\qquad \text{for } 1 \le i \le k, 1 \le j \le n_i-1.\end{equation}
Using above equations (\ref{E:fstar1}) --- (\ref{E:fstar3}), we obtain a matrix representation of $F_*$ on $\pic(X)$.

\subsubsection{Orbit Data} Let us use the ordered basis $\langle e_0, e_{1,1}, e_{1,2}, \dots,e_{1, n_1}, e_{2,1}, \dots, e_{k,1}, \dots, e_{k,n_k} \rangle$. 
Except for the columns corresponding $e_0$ and $e_{i,n_i}$ for $1 \le i \le k$, all other columns of the matrix representation of $F_*$ consist only of zeros, except for a $1$ in the sub-diagonal entry. 

Let a $(k+1) \times (k+1)$ matrix $M$ be a sub-matrix of $F_*$ with columns corresponding to $e_0, e_{i, n_i}$ and rows corresponding to $e_0, e_{i,1}$ for $1 \le i \le k$. Let $\bar{n} = (n_1, n_2, \dots, n_k)$ be a $k$-tuple of strictly positive integers such that $n_i = |\mathcal{O}_i|$. It is clear that we can reconstruct $F_*$ using $M$ and $\bar n$. 
\begin{defn}
With the notations above, we call a $(k+1)\times (k+1)$ matrix $M$ the \textit{orbit permutation} and a $k$-tuple $\bar n$ of positive integers the \textit{orbit lengths}. We call the pair $(M, \bar n)$ the \textit{Orbit data} $\mathcal{O}$ of $f$: \[ \mathcal{O}(f) \ = \ (M, \bar n).\]
\end{defn}

A matrix $M$ in the orbit data of a birational map $f \in \bptwo$ is called the orbit permutation because it describes the correspondence between the exceptional curves and the points of indeterminacy. Note that the orbit data is uniquely determined if the order of exceptional curves is given. 

\begin{prop}\label{P:nozeros}
Suppose a birational map $f \in \bptwo$ has the orbit data $\mathcal{O}(f) = (M, \bar n)$. Then the first column and the first row of $M$ consist of non-zero integers. 
\end{prop}

\begin{proof}
This follows directly from equations (\ref{E:fstar1}) - (\ref{E:fstar2}).
\end{proof}

\begin{thm}\label{T:bOrbitD}
The dynamical degree of a birational map $f \in \text{Bir}(\mathbf{P}^2(\mathbb{C}))$ is determined by its orbit data $\mathcal{O}(f)$. 
\end{thm}

\begin{proof} The orbit data $\mathcal{O}(f)$ gives $F_*:\pic(X) \to \pic(X)$ for an algebraically stable modification $F$. By Diller and Favre\cite{Diller-Favre:2001}, we have $\lambda(f) = $ the spectral radius of $F_*$. 
\end{proof}

\begin{eg}\label{Eg:crem}
Suppose $\sigma$ is the Cremona involution \[ \sigma[x_1:x_2:x_3] = [x_2 x_3:x_1x_3:x_1 x_2] .\]
Then we have \[ \mathcal{O}(\sigma) = \left (M=\begin{bmatrix} 2&1&1&1\\-1&0&-1&-1\\-1&-1&0&-1\\-1&-1&-1&0 \end{bmatrix}, \bar n=(1,1,1)\right) \]
\end{eg}

\begin{eg}\label{eg:bp2} Let $L \in \text{Aut}(\mathbf{P}^2(\mathbb{C}))$ and let $ f=L \circ \sigma \in \bptwo$. Then we have 
\[ \mathcal{E}(f) = \{ E_1=\{ x_1=0\}, E_2=\{ x_2=0\}, E_3=\{ x_3=0\} \} \] and \[ \mathcal{I}(f) = \{ p_1=[1:0:0],p_2=[0:1:0],p_3=[0:0:1] \} \] 
Now suppose that there is a permutation $\rho \in S_3$, $n_1,n_2,n_3 \ge 1$ such that
\[ f^{n_i+1}(E_i) = p_ {\rho(i)}, \qquad f^j (E_i) \not\in \mathcal{I}(f) \text{ for all } 1 \le j \le n_i.\]
Then we have $\mathcal{O}(f) = (M_\rho = M \cdot m(\rho), (n_1,n_2,n_3))$ where $M$ is the matrix given in the previous example and $m(\rho)$ is the matrix representation corresponding to $\rho \in S_3$. For example, if $\rho=(1,2)$ is a transposition of $1$ and $2$, then 
\[ m((1,2)) = \begin{bmatrix} 1&0&0&0\\ 0&0&1&0\\ 0&1&0&0\\0&0&0&1\end{bmatrix}, \quad \text{and} \quad M \cdot m((1,2)) = \begin{bmatrix} 2&1&1&1\\-1&-1&0&-1\\-1&0&-1&-1\\-1&-1&-1&0 \end{bmatrix}. \]
\end{eg}

\begin{eg}\label{Eg:tau}
Consider the involution $\tau$ defiend by \[ \tau[x_1:x_2:x_3] = [x_1^2: x_1 x_2: x_2^2-x_1 x_3]. \]
Since $\det J_\tau = -2 x_1^3$, we have \[  \mathcal{E}(\tau) =   \{ E_1=\{ x_1=0\}, E_2=\{ x_1=0\}, E_3=\{ x_1=0\} \},\qquad \mathcal{I}(\tau) = \{p= [0:0:1] \}. \]
Notice that $\tau^2[x_1:x_2:x_3] = [x_1^4:x_1^3 x_2: x_1^3 x_3] = [x_1:x_2:x_3]$, and thus $\tau^2 = Id$. It follows that \[\dim \tau (E_i) = 0, \qquad \text{and} \qquad \dim \tau^2(E_i) = 1 \qquad \forall i=1,2,3. \] Thus we have $n_1=n_2=n_3 = 1$. 
The orbits of exceptional curves are $\mathcal{O}_i = \{  p=[0:0:1]\} $ for all $i=1,2,3$. Let $E= \{ x_1=0\}$ be the common set for exceptional curves. 
Let us define a modification $\pi:X= X_3\xrightarrow{\pi_3} X_2\xrightarrow{\pi_2} X_1\xrightarrow{\pi_1}X_0 = \ppc$ as follows:
\begin{enumerate}\addtolength{\itemsep}{1ex}
\item $\pi_1: X_1 \to X_0= \ppc$ is a blowup of $\ppc$ at a point $p_1=[0:0:1] \in \mathcal{O}_1$. Let $\tilde E_1 = E$, then we have $p_1\in \tilde E$. 
\item Let $\mathcal{F}_1$ be the exceptional curve over $p_1$.
\item Let $f_1: X_1 \dasharrow X_1$ be the lift of $f$ on $X_1$.
\item Let $\tilde{E_2}\subset X_1$ be the strict transform of $E$. 
\item Let $p_2 = f_1(\tilde{E_2}) \in \mathcal{F}_1$. In fact, we have $p_2 = \mathcal{F}_1 \cap \tilde{E_2}$ and $\pi_1(p_2) = p$. 
\item $\pi_2: X_2 \to X_1$ is a blowup at a point $p_2$, and let $\mathcal{F}_2$ be the exceptional curve over $p_2$.
\item Let $f_2: X_2 \dasharrow X_2$ be the lift of $f$ on $X_2$.
\item Let $\tilde{E_3}\subset X_2$ be the strict transform of $E$. 
\item Let $p_3 = f_2(\tilde{E_3}) \in \mathcal{F}_2$. In fact, we have $p_3 \in \mathcal{F}_2 \setminus E_3$. 
\item Let  $\pi_3: X_3 \to X_2$ is a blowup at a point $p_3$, and let $\mathcal{F}_3$ be the exceptional curve over $p_3$.
\item Finally,  let $\hat{E_i} \subset X$ be the total transform of $\tilde {E_i}$ for $i=1,2,3$: \[ \hat E_3 = \pi_3^* \tilde E_3,\ \ \hat E_2 = (\pi_2 \circ \pi_3)^* \tilde E_2,\ \ \hat E_1 = (\pi_1\circ \pi_2 \circ \pi_3)^* \tilde E_1.\]
\end{enumerate}
It follows that the total transformation $\mathcal{E}_i$ of $p_i$ is given by $\cup_{j=i}^3 \mathcal{F}_j$:
\[ \mathcal{E}_1 =  \mathcal{F}_1+  \mathcal{F}_2 + \mathcal{F}_3, \qquad \mathcal{E}_2 =  \mathcal{F}_2 + \mathcal{F}_3, \qquad \text{and} \qquad \mathcal{E}_3 =   \mathcal{F}_3,\]
 and 
\[ \hat{E}_1= \hat{E}_3 +  \mathcal{F}_1+ \mathcal{F}_2,\qquad \text{and} \qquad  \hat{E}_2 =  \hat{E}_3 + \mathcal{F}_2.\]
Then, the lift $F$ of $f$ on $X$ is an automorphism such that 
\[ F: \hat{E_i} \mapsto \mathcal{E}_i \qquad \text{for all } i=1,2,3 .\]
The Picard group $Pic(X)$ is generated by $e_0,e_1,e_2,e_3$ where $e_0$ is the class of generic line and $e_i$ is the class of $\mathcal{E}_i$ for $i=1,2,3$. With the ordered basis $\langle e_0,e_1,e_2,e_3\rangle$, we have 
\[\mathcal{O}(\tau ) = (M = M , (1,1,1))\] where $M$ is the matrix given in the example \ref{Eg:crem}. Notice that the orbit data for $\tau$ is the same as that for the Cremona involution $\sigma$ in the example \ref{Eg:crem}.
\end{eg}

\begin{rem}
As demonstrated in the example above, computing orbit data for a given birational map can be computationally expensive. We will address this challenge by utilizing orbit data from the Weyl group, as defined in Section \ref{S:coxeter}.
\end{rem}

\subsection{Dynamical Spectrum}
Let us denote by $\Lambda_\mathbb{C} $ the Dynamical spectrum, the set of all dynamical degrees of birational maps on $\mathbf{P}^2(\mathbb{C})$:
 \[\Lambda_\mathbb{C} = \{ \lambda(f) : f \in \text{Bir}(\mathbf{P}^2(\mathbb{C})) \}. \] 
 There are only three possibilities for the dynamical degrees in $\Lambda_\mathbb{C}$. 
 By Diller-Favre \cite{Diller-Favre:2001} and Blanc-Cantat \cite{blancdynamical}, we know: 
\begin{thm}\cite{Diller-Favre:2001,blancdynamical}\label{T:df}
Suppose $f$ is a birational map on a surface $\mathbf{P}^2(\mathbb{C})$. If $\lambda(f) \ne 1$, then 
\begin{itemize}
\item either the dynamical degree is a Salem number or a reciprocal quadratic integer and $f$ is birationally equivalent to an automorphism:  \[ \Lambda_S: =\{ \lambda(f) :  f \in \text{Bir}(\mathbf{P}^2(\mathbb{C}))\text{ and } f \ \text{is birationally equivalent to an automorphism} \} \]
\item or the dynamical degree is a Pisot number, and $f$ is not birationally equivalent to an automorphism: \[ \Lambda_P: =\{ \lambda(f) :  f \in \text{Bir}(\mathbf{P}^2(\mathbb{C}))\text{ and } f \ \text{is NOT birationally equivalent to an automorphism}. \} \]
\end{itemize}
\end{thm}
A Salem number $\lambda_s$ is an algebraic integer greater than $1$ with a monic, self-reciprocal minimal polynomial of degree $\ge 4$, such that all Galois conjugates except $1/\lambda_s$ lie on the unit circle. A Pisot number $\lambda_p$ is an algebraic integer greater than $1$ such that all Galois conjugates have a modulus strictly less than $1$. 

\vspace{1ex}
Suppose $f\in \bptwo$ is birationally equivalent to an automorphism with $\lambda(f)>1$, and let $F: X \to X$ be its algebraically stable modification on a rational surface $X$. Then $F$ is an automorphism on a rational surface $X$, and by Nagata, \cite{Nagata, Nagata2}, $F_*$ belongs to $W_n$, the Coxeter group corresponding to the $E_n$ diagram where $n+1$ is the Picard number of $X$. We will discuss the Coxeter group in the next Section \ref{S:coxeter}. 

The Converse is not true. Kim \cite{Kim:2022} shows that there is an element $\omega \in E_{14}$ with $\rho(\omega)>1$ such that there is no rational surface automorphism $F$ such that $F_* = \omega$. However, McMullen \cite{McMullen:2007}, Uehara \cite{Uehara:2010}, and Kim \cite{Kim:2023} show that for any spectral radius of $\omega \in E_n$ with $ n\ge 10$, there is a birational map $f\in \bptwo$ such that $\lambda(f) = \rho(\omega)$. 

For a real number $\lambda  \ge 1$, we say $\lambda$ is \textit{realized} if there is a birational map $f \in \bptwo$ with $\lambda(f) = \lambda$.

\section{Coxeter Group}\label{S:coxeter}
Let $\mathbb{Z}^{1,n}, n\ge 3$, be a lattice with the ordered basis $(e_0,e_1, \dots, e_n)$ and the symmetric bilinear form given by \[ e_0 \cdot e_0 = 1, \quad  \ e_i \cdot e_i = -1, \ \text{for } i =1, \dots, n, \quad \text{ and } \quad e_i \cdot e_j =0 \ \text{for } i \ne j. \] Thus, for any $x= (x_0, x_1, \dots, x_n) = \sum x_i e_i \in \mathbb{Z}^{1,n}$, \[ x \cdot x \ = x_0^2 - x_1^2-x_2^2- \cdots - x_n^2.\]
The canonical vector is $\kappa_n = (-3,1,\dots,1) \in \mathbb{Z}^{1,n}$, and let $L_n$ denote the orthogonal complement of $\kappa_n$. 

Also, define \[ \alpha_0 = e_0-e_1-e_2-e_3, \qquad \text{and} \qquad \alpha_i = e_i- e_{i+1},\ i=1, \dots, n-1.\] Simple computations show
\[ \alpha_i \cdot \kappa_n =0, \ \ \alpha_i \cdot \alpha_i = -2, \ \ \ \text{for all } i = 0,1,\dots, n-1.\]
Thus, we see that $L_n$ is generated by these vectors $\alpha_0,\alpha_1, \dots, \alpha_{n-1}$: \[ L_n = \kappa_n^\perp= \langle \alpha_0, \alpha_1, \dots, \alpha_{n-1} \rangle. \]
For each $i=0,\dots, n-1$, define a reflection $s_i$ through $\alpha_i$ on $\mathbb{Z}^{1,n}$:
\[ s_i(x)\  :=\ x + (x \cdot \alpha_i)\, \alpha_i.\]
Then we have: \[ s_i(\kappa_n) = \kappa_n, \quad s_i(\alpha_i) = - \alpha_i \ \ \text{for all } i=0,1,\dots, n-1,  \]
and for $x, y \in L_n$ \[ s_i(x) \in L_n, \qquad s_i(x) \cdot s_i(y) \ = x \cdot y.\]
Thus, $s_i \in O(L_n)$ for $i=0,\dots, n-1$ and satisfy the following relations $(s_i s_j)^{m_{ij} }= 1$ where 
\begin{equation}\label{E:reflections}
m_{ij}\ =\ \left \{ \begin{aligned}
& 1 \qquad \text{if } i=j\\
& 3 \qquad \text{if } \{i,j\} \in \{ \{0,3\}, \ \{i,i+1\}, \ i=1,2, \dots, n-2\} \\
& 2 \qquad \text {otherwise} \\
\end{aligned} \right.
\end{equation}
The subgroup $W_n \subset O(L_n)$ generated by reflections $s_0, \dots, s_{n-1}$ is called the Coxeter group.  
For $m \le n$, we may consider $W_m$ as a subgroup of $W_n$ by fixing $e_i$ for $i \ge m+1$:
\[ W_m \cong  \langle s_0, s_1, \dots, s_{m-1} \rangle  \subset W_n.  \]
With this identification, we consider the sequence \[W_3 \subset W_4 \subset W_5 \subset \cdots W_n \subset W_{n+1} \subset \cdots.\] We define \textit{the infinite Coxeter group} $W_\infty$ as a nested union of $W_n$ for $n\ge 3$:
\[ W_\infty :=\cup_{n\ge 3} W_n. \]

\subsection{Root system}\label{SS:root}
Let $V_n$ be an inner product space over $\mathbb{R}$ with the basis $(\alpha_0, \dots, \alpha_{n-1})$, where \[ (\alpha_i, \alpha_j) = -2 \cos (\pi/m_{ij}). \]  Thus, $(\alpha_i, \alpha_j) = 2,0,-1$ if $m_{ij}= 1,2,3$, respectively. For any $\alpha \in V_n$ with $(\alpha, \alpha) = \pm 2$, we can define a reflection $s_\alpha$ by \[ s_\alpha (\lambda) \ =\ \lambda- \frac{2 (\lambda, \alpha)}{(\alpha, \alpha)} \alpha.\] 
Thus, we may assume that $L_n \subset V_n$ and each $s_i \in W_n$ gives the reflection in $O(V_n)$ preserving $L_n$. 

The generating vectors $\alpha_i$ are referred to as \textit{simple roots}, and the \textit{root system} $\Phi_n = \cup_i W_n \alpha_i$ is the set of $W_n$ orbits in $V_n$ of simple roots. (General theories and precise definitions can be found in many sources, including \cite{Humphreys:1990}.) It is useful to note that any root is either positive or negative; in other words, any root can be expressed either as a positive sum of simple roots or a negative sum.  
\[ \Phi_n = \Phi_n^+ \cup \Phi_n^-\]
where $\Phi_n^\pm = \{ \sum c_i \alpha_i : \pm c_i \ge 0 \, \text{for all } i \} $. 

\subsection{Weyl degree}\label{SS:wdeg}

For  each  element $\omega \in W_\infty$, let us define \textit{Weyl degree}(or simply degree), $\deg(\omega)$, of $\omega$  by
\[ \deg (\omega) : = \omega e_0 \cdot e_0.\] If $\deg (\omega) =1$, then $\omega e_0 = e_0$, and $\omega$ permutes $e_i$'s. Thus, the spectral radius $\lambda(\omega) = 1$. 

For a positive integer $d\ge 1$, let $\Omega_d$ denote the set of elements in $W_\infty$ with the Weyl degree $d$:  \[ \Omega_d \ := \ \{ \omega \in W_\infty \ :\ \deg(\omega)= d \} . \]
Also, let $\Lambda_d$ denote the set of spectral radii of elements in $\Omega_d$: \[ \Lambda_d := \{ \lambda(\omega):  \deg(\omega) =d \}.\]
Clearly, the set of all spectral radii of $W_\infty$ is equal to the union  $\cup_d \Lambda_d$:
\[ \{ \lambda(\omega):  \omega\in W_\infty \} \ =\ \{ \lambda(\omega):  \omega\in \cup_{d\ge 1} \Omega_d \}\  =\  \cup_d \Lambda_d \]
 
McMullen \cite{McMullen:2002} showed that the characteristic polynomial of any element in $W_n$ is either a product of cyclotomic polynomials, a product of a reciprocal quadratic polynomial and cyclotomic polynomials, or a product of a Salem polynomial and cyclotomic polynomials. Thus, the spectral radius is either $1$, a reciprocal quadratic integer, or a Salem number. By Theorem \ref{T:df} of Diller-Favre and Blanc-Cantat, we have \[ \cup_d \Lambda_d = \{1\} \cup \Lambda_S,\] where $\Lambda_S$ is all dynamical degrees greater than $1$ of birational maps equivalent to an automorphism. 

Let us denote $\Lambda_W$, the Weyl spectrum, as the closure of the set of all spectral radii of $W_\infty$:
 \[\Lambda_W \  :=\  \overline{\cup_d \Lambda_d}. \] 
 
\subsection{Matrix Representations}
Let $\omega \in W_n$ such that 
\begin{equation}\label{E:mat} \begin{aligned} &\omega(e_0) = d e_0 - \sum_{i=1}^{n}  c_i e_i,\\ & \omega(e_j) = d_i e_0 - \sum_{i=1}^n m_{i,j} e_i \qquad j = 1, \dots, n, \end{aligned} \end{equation}
where $d= \deg(\omega)$ and $d_i, c_i, m_{i,j}$ are non-negative integers.

Let $\bar d = (d_1, \dots, d_n)$, $\bar c= (c_1, \dots, c_n)$ be $n$-tuples of non-negative integers. Also, let $m= \{ m_{i,j}\}$ be an $n\times n$ matrix with non-negative integers. Then, with an ordered basis $\{ e_0, e_1, \dots, e_n\}$, we can write $\omega$ as an $(n+1) \times (n+1)$ matrix:  
\[ \omega = \begin{bmatrix} d&\bar d \\ -\bar c & -m \end{bmatrix}. \]

\vspace{1ex}
Because $\omega \in O(L_n)$, we have the following basic facts (known as Noether's equalities and inequalities). 
\begin{lem}\label{L:facts}
Suppose $\omega  \in W_n$ is given by (\ref{E:mat}).Then we have: 
\begin{enumerate}
\item $3 (d-1)  = \sum_{i}d_i= \sum_i c_i$.
\item $d^2 -1 =  \sum_{i} d_i^2 = \sum_i c_i^2$.
\item $d_j^2 - \sum_i m_{i,j}^2 =-1$ for all $j=1, \dots, n$. Therefore: 
\begin{itemize}
\item if $d_j >0$, then $ m_{ij}< d_j$ for all $i=1, \dots, n$, 
\item if $d_j=0$, then there exists $1\le i_j\le n $ such that $m_{i_j,j}=1$, and $m_{i,j}=0$ for $i\ne i_j$.
\end{itemize}
\item $c_i^2 - \sum_j m_{i,j}^2 =-1$ for all $i=1, \dots, n$. Therefore:
\begin{itemize}
\item if $c_i >0$, then $ m_{ij}< c_i$ for all $j=1, \dots, n$, 
\item if $c_i=0$, then there exists $1\le j_i\le n$ such that $m_{i,j_i}=1$, and $m_{i,j}=0$ for $j\ne j_i$.
\end{itemize}
\item The sum of the three largest $d_i$'s is greater than or equal to $d+1$, and the same is true for $\bar c$:\[ d_1+d_2+d_3 \ge d+1\ \  \text{after rearranging indices so that  } d_1 \ge d_2\ge \cdots \ge d_n. \]
\item The sum of any two $d_i$'s is smaller than or equal to $d$, and the same is true for $\bar c$: \[ d_i + d_j \le d \ \ \ \text{for }  i \ne j. \]
\item $\omega^{-1}$ is given by the matrix representation: 
\[ \begin{bmatrix} d&\bar c \\ -\bar d & -m^t \end{bmatrix}, \] where $m^t$ is the transpose of $m$. 
\end{enumerate}
\end{lem}

\begin{proof}
Except for the last statement, the proof is written in several articles, including \cite{blancdynamical} and \cite{dolgachev:1883}.
The last statement can be checked by multiplying two matrices using other statements in this lemma. 
\end{proof}
\subsection{Decompositions}
For any $\omega \in W_n$, we have $\omega e_j \cdot \omega e_j =-1$ for $j = 1, \dots, n$. It follows that for each $j = 1, \dots, n$ we have: \[ \text{either }\ \ \omega e_j = e_k\ \ \text{for some  }k \ge 1\ \quad \text{or }\quad \omega e_j \cdot e_0 \ne 0. \]
For each $\omega \in W_n$, let us define: 
\begin{equation}\label{E:twoindex} \begin{aligned} &I_e(\omega) \ :=\  \{ 0 \le i \le n :  \omega^k e_i \cdot e_0 \ne 0 \quad \text{for some } k\ge 0 \}, \\ &I_p(\omega)  \ :=\  \{ 0 \le i \le n :  \omega^k e_i \cdot e_0 = 0 \quad \text{for all } k \ge 0 \}.\\ \end{aligned}\end{equation}
Clearly, if $\omega e_ i = e_j$ with $i \in I_e$, then $j \not\in I_p$. Also, if $i \in I_p$, then $\omega e_i = e_j$ for some $j \in I_p$.  
Let us denote $n_e : = |I_e(\omega)|-1$. Then, by reordering the $e_i's$, we may assume that: \[ I_e(\omega) = \{ 0, 1, \dots, n_e\} , \quad \text{and} \quad I_p(\omega)= \{ n_e+1, \dots, n\}.\]
Let us define $\omega_e \in W_n$ satisfying: \begin{equation}\label{E:twoele1}  \omega_e e_i \ =\  \omega e_i , \text{ for } i \in I_e, \qquad \text{and} \qquad  \omega_e e_i \ =\  e_i, \text{ for } i \in I_p.\end{equation}
Also, let us define $\omega_p \in W_n$ by: \begin{equation}\label{E:twoele2} \omega_p e_i \ =\   e_i , \text{ for } i \in I_e, \qquad \text{and} \qquad  \omega_p e_i \ =\  \omega e_i, \text{ for } i \in I_p.\end{equation}


\begin{lem}\label{L:essentialpart}
Suppose $\omega \in W_n$ such that $I_e(\omega) = \{0,1, \dots, m\}$ for some $m\le n$. Then, with $\omega_e, \omega_p$ given in (\ref{E:twoele1}) and (\ref{E:twoele2}), we have \[ \omega = \omega_e \cdot \omega_p,\] such that:
\begin{enumerate}
\item $\omega_e  \in \langle s_0, s_1, \dots, s_{m-1} \rangle$,  $\deg(\omega) = \deg(\omega_e)$, 
\item $\omega_p \in \langle s_{m+1}, \dots, s_{n-1} \rangle $, and
\item the spectral radius $\lambda(\omega)$ of $\omega$ is given by the spectral radius of $\omega_e$, \[ \lambda(\omega) \ = \ \lambda(\omega_e).\]
\end{enumerate}
\end{lem}

\begin{proof}
From equations (\ref{E:twoele1}) and (\ref{E:twoele2}), we see that $\omega = \omega_e \cdot \omega_p$, which validates the first two assertions. Additionally, it is evident that the matrix representation of $\omega$ can be written as a block diagonal matrix, and $\lambda(\omega_p) =1$. Thus, we conclude that $ \lambda(\omega) \ = \ \lambda(\omega_e)$.
\end{proof}

Since $\omega_p$ does not contribute to the spectral radius of $\omega$, let us focus on elements with $\omega_ p = Id$. 
By reordering $e_1, \dots, e_n$, there are $k \le n$ distinct positive integers, $1\le n_1< n_2< \cdots < n_k \le n$, such that: 
\begin{itemize}
\item $\omega e_i \ = \ e_{i+1}$ if  $i \not\in \{ n_j, j=1, \dots, k \}$, 
\item $\omega e_{n_j} \cdot e_0 = d_j >0$ for all $ j = 1, \dots, k$.
\end{itemize}

Let us set $n_0 = 0$ and define $\omega_c \in W_n$ by:
\begin{equation}\label{E:wc} \begin{aligned} \omega_c \ :\  & e_0 \mapsto \omega e_0,\\ &e_{n_{i-1}+1} \mapsto \omega e_{n_i}, i =1, \dots, k, \quad  \text{and} \\ & e_j \mapsto e_j ,  \text{ for } j \ne n_i+1, i =0,1, \dots, k-1. \\ \end{aligned}.  \end{equation}
 Also, let us define $\omega_s \in W_n$ by: 
 \begin{equation}\label{E:ws} \begin{aligned} \omega_s \ :\ & e_0 \mapsto e_0, \\ & e_1 \mapsto e_2 \mapsto \cdots \mapsto e_{n_1} \mapsto e_1,  \quad \text{and}\\ & e_{n_{i-1} + 1} \mapsto  e_{n_{i-1} + 2} \mapsto \cdots \mapsto  e_{n_{i} }\mapsto  e_{n_{i-1} + 1}, \ \ \text{for } i = 2, \dots, k.\\ \end{aligned} \end{equation}
With Lemma \ref{L:essentialpart}, we see that $\omega \in W_n$ is conjugate to a product of three elements defined in (\ref{E:twoele1} - \ref{E:ws}). 

\begin{prop}\label{P:decom}
Suppose $\omega \in W_n$ such that $\deg(\omega) = d \ge 2$. Then, there is a permutation $\rho$ acting on $e_i,  i=1, \dots, n,$ such that
\[ \omega = \rho^{-1} \cdot \omega_c \cdot \omega_s \cdot \omega_p \cdot \rho. \]
Furthermore, we have $\deg(\omega_c) = \deg( \omega)$ and $\lambda(\omega_c \cdot \omega_s) = \lambda(\omega)$.
\end{prop}

\begin{proof}
The reordering of the elements $e_i$'s, $i=1, \dots, n,$ is achieved by conjugating by a permutation in $e_i,  i=1, \dots, n$. Thus, this proposition follows from Lemma \ref{L:essentialpart}.
\end{proof}

\subsection{Formulas for Spectral radii}\label{SS:fspec}
In this subsection, we obtain a formula for the characteristic polynomial of a given $\omega \in \Omega_d$ up to cyclotomic factors. Since $\omega_p$ in Proposition \ref{P:decom} relates only to the cyclotomic factor, we may assume both $\omega_p$ and $\rho$ in Proposition $\ref{P:decom}$ are the identity. 

\vspace{1ex}
For $\omega = \omega_c \cdot \omega_s \in \Omega_d$, let $\tilde \omega_c$ be a $(k+1) \times (k+1)$ submatrix of $\omega_c$ given by columns and rows corresponding to
 $ e_0, e_{n_0+1}, \dots ,e_{ n_{k-1}+1}$. In other words, $\tilde \omega_c$ is a matrix obtained by excluding columns and rows corresponding to the $e_i$'s fixed by $\omega_c$. Additionally, let $n(\omega)$ be a $k$ tuple of positive integers representing the lengths of cycles in $\omega_s$
\[ n(\omega)\  :=\  (n_1, n_2-n_1, n_3-n_2, \cdots, n_k-n_{k-1} ).\]
For each $\omega \in W$, the numerical data $k, \tilde \omega_c,$ and $n(\omega)$ are uniquely determined up to reordering indices. We define the \textit{orbit data} $\mathcal{O}(\omega)$ of $\omega$ by the numerical data $(\tilde \omega_c, n(\omega))$: \[ \mathcal{O}(\omega) : = (\tilde \omega_c, n(\omega)).\] Also, we refer to $n(\omega)\in \mathcal{O}(\omega)$ as the \textit{orbit lengths}. The element $\tilde \omega_c$ is determined by the images of $e_i$, which map to a vector $\ne e_j,$ for any $j\ge 1$ under $\omega$. We call $\tilde \omega_c$ the \textit{orbit permutation}. If $v = \sum m_i e_i$ satisfies $v\cdot v = -1$ and $v \ne e_j$ for all $j\ge 1$, then $m_0 \ne 0$. Thus, we have: 
\begin{prop}\label{P:nozerosw}
Suppose $\omega \in W_n$ has the orbit data $\mathcal{O}(\omega) = (\tilde \omega_c, \bar n)$. Then, the first column and the first row of $\tilde \omega_c$ consist of non-zero integers. 
\end{prop}

\begin{rem}
This definition of orbit data for an element of a Coxeter group is comparable with the orbit data of a birational map defined in Section \ref{S:birational}.
If $f \in \text{Bir}(\mathbf{P}^2(\mathbb{C}))$ is birationally equivalent to an automorphism $F:X \to X$ on a rational surface $X$, then the orbit data of $f$ is equal to the orbit data of the linear action $F_* \in W_\infty$ up to relabelling the $e_i$'s.
\end{rem}

\begin{lem} For $\omega \in \Omega_d, d\ge 2$, the spectral radius $\lambda(\omega)$ of $\omega$ is determined by the orbit data, a $(k+1)\times (k+1)$ matrix $\tilde \omega_c$, and a $k$-tuple $n(\omega)$ of positive integers. 
\end{lem}

\begin{proof}
From (\ref{E:wc}) and (\ref{E:ws}), it is clear that we can reconstruct $\omega_c$ and $\omega_s$ from the orbit data. 
\end{proof}

For $d \ge 2$, let us define $\mathcal{M}_d$ as the set of all orbit permutations, $\tilde \omega_c$ for all $\omega \in \Omega_d$:  \[ \mathcal{M}_d:= \{ \tilde \omega_c : \omega \in \Omega_d \}.\] 

\begin{lem} Let $d \ge 2$. The set of orbit permutations $\mathcal{M}_d$ is a finite set:
\[ |\mathcal{M}_d| < \infty. \]
\end{lem} 

\begin{proof}
Suppose $\begin{bmatrix} d&\bar d\\-\bar c&-m\end{bmatrix} \in \mathcal{M}_d$. Then $d_i,c_i \ne 0$ for all $i$. By Lemma \ref{L:facts}, we have: 
\[ 3(d-1) = \sum_{i} d_i,\quad \text{and} \quad d^2-1 = \sum_{i} d_i^2. \]
Since $d_i$ are all positive integers, there are only finitely many possible solutions.
Also, by Lemma \ref{L:facts}, we observe that $0\le m_{ij} \le d_j$ for all $i,j$. It follows that $\mathcal{M}_d$ is a finite set. 
\end{proof}

Suppose $M$ is a $(k+1)\times (k+1)$ matrix and $\bar n=(n_1, \dots, n_k)$ is a $k$-tuple of positive integers. Let $N=\sum_{i=1}^k n_i$. For a subset $I \subset \{1, 2, \dots, k\}$, let $n_I = \sum_{i \in I} n_i$ and let $M_I$ be a principal submatrix of $M$ obtained by deleting the $(i+1)^\text{th}$ column and row for all $i \not\in I$. For example, if $I=\{1,2, \dots,k\}$, then $M_I = M$, and if $I= \emptyset$, then $M_\emptyset$ is an $1\times1$ matrix whose single entry is the $(1,1)$ entry of $M$. 

\vspace{1ex}
Let $M_t$ be a matrix obtained from $M$ by subtracting $t$ from the $(1,1)$ entry. 
Then, for $I \subset \{1, \dots, k\}$, $M_{t,I}$ is a principal submatrix of $M_t$ obtained by deleting the $(i+1)^\text{th}$ column and row for all $i \not\in I$ and the determinant of $M_{t,I}$ is a degree one polynomial in $t$. 
Let us define a polynomial:
\begin{equation}\label{E:charpoly}
\chi_{M, \bar n} (t) := \sum_{I \subset \{1, \dots, k\}} (-1)^{|I|-1}\, t^{\sum_{i \not\in I} n_i}\, \det M_{t,I}.
\end{equation}

Let $I(\bar n) = \{ i+j : 1 \le j \le n_i, i=1, \dots, k\}$, $N= \sum_{i=1}^k n_i$, and let  $M(\bar n)$ be a $(N+k+1)\times (N+k+1)$ matrix such that 
\begin{itemize}
\item A sub-matrix of $M(\bar n)$ obtained by removing the $i^\text{th}$ rows and $i^\text{th}$ columns for $i \in I(\bar n)$ is identical to $M$, and
\item for $i \in I(\bar n)$, the $i^\text{th}$ column of $M(\bar n)$ has $1$ in the $i+1^\text{th}$ entry and $0$ otherwise.
\end{itemize}
Notice that this matrix $M(\bar n)$ corresponds to the matrix representation of $\omega$ with the orbit data $\mathcal{O}(\omega) = (M, \bar n)$.

\begin{lem}\label{L:charpoly} With the above settings, the characteristic polynomial of $M(\bar n)$ is given by $\chi_{M, \bar n}(t)$ as defined in (\ref{E:charpoly}).
\end{lem}
 
 \begin{proof}
Suppose $k=1$ and $n_1=1$. Then we have \[ \omega - t I = \begin{bmatrix} d-t&d_1\\ -c_1& -m_{1,1}-t \end{bmatrix} \]
and thus \[ \det(\omega-t I) = -t (d-t) + \det \begin{bmatrix} d-t&d_1\\ -c_1& -m_{1,1} \end{bmatrix}. \]
If $n_1=n>1$, then we have:
\[ \omega - t I = \begin{bmatrix} d-t&\ 0\ &\ 0\ & &&&d_1\\ -c_1&-t &\  0\ & &&&-m_{1,1}\\ \ 0\ &\ 1\ &-t& &&&\ 0\ \\ \ 0\ &\ 0\ &\ 1\  & -t  &&& \\ \ 0\ &\ 0\ &\ 0\  &\ 1\  &-t&&\\ \ 0\ &  & &&&&\\  \ 0\ &  & &&\,\ddots&\,\ddots & \\ \\  \ 0\ & &&\ \ &&\ \ \ 1&-t\end{bmatrix}. \]
Then, the principal submatrix obtained by deleting the $2^\text{nd}$ column and the $2^\text{nd}$ row has a row with a single element $-t$ in the diagonal, and thus, using minors corresponding to $-t$ in such a row, we see that the cofactor corresponding to $(2,2)$-entry is $(-t)^{n-1} (d-t)$. Similarly, we use minors corresponding to a single $1$ in a row for the minor of $(2,3)$-entry. Then we see that the cofactor corresponding to $(2,3)$-entry is $(-1)^{n-1} \det \begin{bmatrix} d-t&d_1\\ -c_1& -m_{1,1} \end{bmatrix}$. It follows that 
 \[ \det(\omega-t I) = (-t)^{n} (d-t) + (-1)^{n-1} \det \begin{bmatrix} d-t&d_1\\ -c_1& -m_{1,1} \end{bmatrix}. \]
 Suppose $k=2$, and $n_1, n_2$ are two strictly positive integers. Then
  \[ \omega - t I = \begin{bmatrix} d-t&|&0&&d_1&|&&&d_2\\-&-&-&-&-&-&-&-&-\\ -c_1&|&-t &  &-m_{1,1}&|&&&-m_{1,2}\\ 0&|&1&\ddots&0 &|& &&\\ 0&|&&1 & -t&|&&&\\-&-&-&-&-&-&-&-&-\\ -c2&|& & &-m_{2,1}&|&-t& & -m_{2,2}\\&|&&&&|&1&\ddots& 0\\&|&  & &&|& &1& -t\end{bmatrix} \]
The $(n_1+1)^\text{th}$ to $(n_1+n_2)^\text{th}$ columns have exactly two non-zero entries: $(-t)$ in the sub-diagonal and $1$ below.  Using these columns for the cofactor expansion, we have:
  \begin{equation}\label{E:nasty}\begin{aligned}
 \det (\omega-t I)  &= (-t)^{n_2} \det \begin{bmatrix} d-t&0& &d_1\\ -c_1&-t &  &-m_{1,1} \\ 0& 1  &\ddots & 0\\ 0& &1&-t\end{bmatrix}\\&+(-1)^{n_2-1}  \det \begin{bmatrix} d-t&0& &d_1&d_2\\ -c_1&-t &  &-m_{1,1}&-m_{1,2} \\ 0& 1  &\ddots & 0&0\\ 0& &1&-t&0\\ -c_2 & & &-m_{2,1}&-m_{2,2}\end{bmatrix}\end{aligned}\end{equation}
  For the second term on the right-hand side of the above equation, we use, again, $(-t)$'s and $1$'s for the cofactor expansion, as in the case with $k=1$. Then we see the second term is equal to \begin{equation}\label{E:nasty2} (-1)^{n_2-1}  \left ( (-t)^{n_1}\det  \begin{bmatrix} d-t&d_2\\-c_2&-m_{2,2} \end{bmatrix} +(-1)^{n_1-1} \det\begin{bmatrix} d-t&d_1&d_2\\ -c_1&-m_{1,1}&-m_{1,2}\\-c_2&-m_{2,1}&-m_{2,2} \end{bmatrix}. \right)\end{equation}
Combining equations (\ref{E:nasty}) and (\ref{E:nasty2}), we see that the characteristic polynomial of $\omega$ is given by the formula in (\ref{E:charpoly}).
For $k\ge 3$, we work with cofactor expansions using $(-t)$'s and $1$'s for each block. The computation is essentially identical. 
\end{proof}



\begin{thm}\label{T:charpoly}
Suppose the orbit data of $\omega$ is given by $\mathcal{O}(\omega)=(M=\begin{bmatrix} d&\bar d\\-\bar c&-m\end{bmatrix}, \bar n=(n_1, \dots, n_k))$. Let $\omega _c$ and $\omega_s$ be elements of $W_\infty$ defiend in (\ref{E:wc}) and (\ref{E:ws}). Then,
the characteristic polynomial of $\omega_c \cdot \omega_s$ is the polynomial $\chi_{\mathcal{O}(\omega)}(t)$ defined in (\ref{E:charpoly}). Thus, the spectral radius of $\omega$ is the largest real root of $\chi_{\mathcal{O}(\omega)}(t)$.
\end{thm}

\begin{proof}
Since the matrix representation of $\omega_c\cdot \omega_s$ is given by $M(\bar n)$, this theorem follows from Lemma \ref{L:charpoly}.
\end{proof}

Let $f\in\bptwo$ be a birational map with the orbit data $\mathcal{O}(f)=(M, \bar n)$. Then there exists a birational map $F$ on a rational surface $X$ such that $F$ is an algebraically stable modification of $f$ with the induced action of $F$ on $\pic(X)$ given by $M(\bar n)$. Thus, we have:

\begin{cor}
Let $f\in\bptwo$ be a birational map with the orbit data $\mathcal{O}(f)=(M, \bar n)$. Then, the dynamical degree of $f$ is the largest real root of $\chi_{M, \bar n}$ defined in (\ref{E:charpoly}).
\end{cor}

\begin{eg}\label{E:noJ} Let $\omega \in \Omega_3$ given by 
\begin{equation}
\omega \ :\ \left\{ \begin{aligned} &e_0 \ \mapsto 3 e_0 -2 e_1-e_4-e_6-e_{10}-e_{11}, \\ & e_1 \mapsto e_2\mapsto e_3 \mapsto e_0-e_1 - e_6,\\ & e_4\mapsto e_5 \mapsto 2 e_0 - e_1-e_4-e_6-e_{10}-e_{11},\\ &e_6 \mapsto e_7 \mapsto e_ 8 \mapsto e_9 \mapsto e_0- e_1 - e_{10}, \\ & e_{10} \mapsto e_0 - e_1 - e_4,\\& e_{11} \mapsto e_{12} \mapsto e_0-e_1- e_{11}. \\ \end{aligned} \right.
\end{equation}
Then we have $k=5$, $n(\omega) =( 3,2,4,1,2)$,
\[M= \begin{bmatrix} d&\bar d\,\\-\bar c&*\,\end{bmatrix} \ =\ \begin{bmatrix} \ 3 & \ 1&\ 2&\ 1 &\ 1&\ 1\\ -2& -1&-1&-1&-1&-1 \\ -1 &\ 0 &-1&\ 0&-1&\ 0\\-1&-1&-1&\ 0&\ 0&\ 0\\-1&\ 0&\ -1&\ -1&\ 0&\ 0 \\ -1&\ 0&-1&\ 0&\ 0&-1\\ \end{bmatrix}. \] 
For $I\subset \{1, \dots, 5\}$, 
\[ M_{t,\emptyset} \ =\  \begin{bmatrix} 3-t \end{bmatrix}, \ \ \  M_{t,\{1\}} \ =\    \begin{bmatrix} 3-t &1\\-2&-1\end{bmatrix}, \ \ M_{t,\{2\}}  \ =\    \begin{bmatrix} 3-t &2\\-1&-1\end{bmatrix} , \ \ M_{t,\{3\}} \ =\    \begin{bmatrix} 3-t &1\\-1&0\end{bmatrix}, \cdots \]   \[ M_{t,\{1,2\}} \ =\    \begin{bmatrix} 3-t &1&2\\-2&-1&-1\\-1&0&-1 \\\end{bmatrix},\ \ M_{t,\{1,3\}} \ =\    \begin{bmatrix} 3-t &1&1\\-2&-1&-1\\-1&-1&0 \\\end{bmatrix}, \ \ \dots \ , \ \ M_{t,\{1, 2, \dots, k\}} = M_t.\]
Using (\ref{E:charpoly}), we see that the characteristic polynomial of $\omega$ is given by: 
\[ \chi(t) =  (t^{12}-2 t^{11} + t^{10} - 2 t^9+t^7-t^6+t^5-2 t^3+t^2-2 t+1) (t-1) \]
and the spectral radius is $1.96683..$, the largest real root of $\chi(t)$. 
\end{eg}

\subsection{Special Case.} 
Blanc and Furter \cite{Blanc-Furter:2019} considered the subgroup $J_q, q\ne 0$ of $W_\infty$ consisting of elements $\omega$ such that
\begin{enumerate}
\item $\omega( e_0 - e_q) = e_0 - e_q$, and 
\item there exists a birational map $f: \ppc\dasharrow \ppc$ with $\mathcal{O}(f) = \mathcal{O}( \omega)$. 
\end{enumerate}
This group $J_q$ corresponds to the group of Jonqui\`{e}res transformations preserving the pencil of lines passing through a point in $\mathbf{P}^2(\mathbb{C})$.

A birational map $f: \ppc \dasharrow \ppc$ with orbit data $\mathcal{O}(f) =( M, \bar n)$ where 
\begin{equation}\label{E:jon}M=  \begin{bmatrix} d& d-1&1& \cdots &1\\ -(d-1) & -(d-2)&-1 & \cdots &-1\\ -1 &-1 & &&\\ \vdots & \vdots& &- Id& \\ -1&-1 & & \end{bmatrix} \end{equation}
is  the homaloidal type of a Jonqui\`{e}res element \cite{Blanc-Furter:2019}. For each $(2d-1)$-tuple $\bar n=(2, n_2, \dots, n_{2d-1})$ of positive integers in $\mathbb{N}^{2d-1}$, Bot \cite{Bot:2024} constructed a Jonqui\`{e}res transformation with orbit data $(M, \bar n)$.

\vspace{1ex}
In this section, we apply our formula in (\ref{E:charpoly}) to a larger subgroup $\Omega \subset W_\infty$ defined by: 
\[ \Omega =\{ \omega \in W_\infty : \mathcal{O} (\omega) = (\omega_d \cdot \sigma, \bar n),  \sigma \in \langle s_1, \dots, s_{2d-2} \rangle, \bar n \in \mathbb{N}^{2d-1}, d \ge 2 \}, \] where
\[ \omega_d = \kappa_{1,3,2}\cdot \kappa_{1,5,4}\cdot \kappa_{1, 7,6} \cdots \kappa_{1, 2d-1, 2d-2}, \]
and $\kappa_{i,j,k} \in W_\infty$ is the reflection through the vector $e_0 - e_i-e_j-e_k$.
\vspace{1ex}
For any $d\ge 2$, we see that  $ \omega_d =M $ defined in (\ref{E:jon}),
and \[\mathcal{O}(\omega_d) = (\omega_d, \bar n = (1, 1, \dots, 1)).\]   Let $\sigma \in \langle s_1, \dots, s_{2d-2}\rangle$ be a permutation in $\{1,2, \dots,2d-1\}$ . Then, we have:  \[\mathcal{O}(\omega_d \cdot \sigma) = (\omega_d\cdot \sigma, \bar n = (1, 1, \dots, 1)).\] 
Note that if $\omega \in \Omega$, $\sigma = Id$, and $n_1 = q\ge 1$, then $\omega(e_0-e_q ) = e_0-e_1$. 

\vspace{1ex}
For any $\omega \in \Omega$, the corresponding $M_{t,I}$'s in (\ref{E:charpoly}) and their determinants are relatively simple. Due to these simplicities, the Salem polynomials for the spectra radii can be obtained using linear algebra. To avoid tedious computations, we will describe the Salem polynomials appearing in $\Omega$. 

\vspace{1ex}
Suppose $d\ge 2$ is a positive integer and $\sigma \in \langle s_1, \dots, s_{2d-2}\rangle$ is a permutation in $\{1, 2, \dots, 2d-1\}$.
For each subset $I \subset \{1,2, \dots, 2d-1\}$,
let us define $\sigma^{-1}(I) = \{ i : \sigma(i) \in I\}$ and 
\[ \rho(I) \ = \ \left \{ \begin{aligned} \ &5 \qquad \text{if  } I = \sigma^{-1}(I), \text{ and } 1 \in I, \\ &4 \qquad \text{if  } |I \setminus \sigma^{-1}(I)| =1, \text{ and } 1 \in I \cap \sigma^{-1}(I), \\&3 \qquad \text{if  } |I \setminus \sigma^{-1}(I)| =1, \text{ and } 1 \in I  \cup \sigma^{-1}(I) , 1 \in I\cap \sigma^{-1}(I), \\&2 \qquad \text{if  } I= \sigma^{-1}(I), \text{ and } 1 \not\in I \cup \sigma^{-1}(I), \\&1 \qquad \text{if  } |I \setminus \sigma^{-1}(I)| =1, \text{ and } 1 \not\in I \cup \sigma^{-1}(I), \\ & 0 \qquad \qquad \text{otherwise}. \end{aligned} \right. \]
Also, let \[ \text{Sgn}(I) = \left\{ \begin{aligned} \ \ &+ \qquad \text{if  }\sigma|_I \text{ is given by a product of an even number of transpositions }, \\ 
&- \qquad \text{if  }\sigma|_I \text{ is given by a product of an odd number of transpositions. }\\  \end{aligned} \right. \]
Here, $\sigma|_I$ is a permutation such that $\sigma|_I(i) = \sigma(i)$ for $i \in I$ and $\sigma|_I(i) =i$ otherwise. 
Finally, let $n = |I|$ and let us define
\[ \psi(I) \ = \ \left\{ \begin{aligned}  \ \ & Sgn(I) ( (n+1-d) t -1) \qquad \ \ \text{  if  } \rho(I) = 5, \\& Sgn(I) t \qquad \qquad  \qquad \qquad  \qquad\text{  if  } \rho(I) = 4, \\& -Sgn(I) (t-1) \qquad \qquad \qquad\  \text{  if  } \rho(I) = 3, \\& -Sgn(I) (t-d+n) \qquad\  \  \qquad  \text{  if  } \rho(I) = 2, \\& -Sgn(I) \qquad  \qquad \qquad \qquad\quad \text{  if  } \rho(I) = 1, \\ & 0 \qquad \qquad \qquad \qquad \qquad \qquad \quad \text{  if  } \rho(I) = 0. \\ \end{aligned} \right. \]

\begin{lem}\label{L:d1s}
Suppose $\omega \in \Omega$ with orbit data $\mathcal{O}(\omega) = ( \omega_d \cdot \sigma, \bar n = (n_1, n_2, \dots, n_{2d-1}) )$. Then, the spectral radius is the largest real root of $\chi(t)$ where \[ \chi(t) = t^N  \left[ (d-t) + \sum_{I \ne \emptyset} \frac{1}{t^I} \psi(I) \right], \]  $N=n_1+n_2+ \cdots + n_{2d-1}$, and  $t^I = t^{ \sum_{i \in I} n_i}$. Furthermore, if $\sigma(1) = 1$ and $n_1=1$, then the spectral radius is equal to $1$ for all choices of $n_i$, $ i\ge 2$.
\end{lem}


\section{Weyl Spectrum vs. Dynamic Spectrum}\label{S:spec}

The dynamical spectrum $\Lambda_\mathbb{C}$ does not contain all Salem numbers, and it is difficult to know all birational maps and their dynamical degrees. However, it is known that the dynamical spectrum is, in fact, the same as the Weyl spectrum, as noted by McMullen \cite{McMullen:2007}, Uehara \cite{Uehara:2010}, and Kim \cite{Kim:2023}. 

Thus, we have:
\begin{prop}\label{P:nosalem}
Suppose $\delta$ is a limit point of $\Lambda_S$. Then either $\delta$ is a Pisot number or $\delta$ is a Salem number in $\Lambda_S$.
\end{prop}

\begin{proof}
Since a limit point of $\Lambda_S$ is either a Pisot number or a Salem number, it is sufficient to consider the case where $\delta$ is a Salem number. 
By Blanc and Canta \cite{blancdynamical}, we know that the dynamical spectrum is closed, and thus $\delta$ is the dynamical degree of a birational map $f$ on $\ppc$. 
Since $\delta$ is a Salem number, a birational map $f$ is equivalent to an automorphism, the induced action on the Picard group is an element $\omega$ of $W_n$ for some $n\ge 10$, and $\delta$ is the spectral radius of $\omega$. 
\end{proof}

\subsection{Monotonicity}\label{SS:mono} In this subsection, we show that for a given $M\in \cup_d \mathcal{M}_d$, the spectral radius of $\omega \in \Omega_d$ is increasing with respect to orbit lengths.

\begin{prop}\label{P:n1}
Suppose $\chi(t)$ is the characteristic polynomial (up to a product of cyclotomic polynomials) of an element $\omega$ with orbit data $\mathcal{O}(\omega)=(M \in \cup \mathcal{M}_d, (n_1,n_2, \dots, n_k))$ defined in (\ref{E:charpoly}). Then, for $i=1, \dots, k$ we have
\[\chi(t) = t^{n_i} \psi_i(t) + \phi_i(t), \quad \text{where}\]
\begin{enumerate}
\item $\psi_i(t)$ and $\phi_i(t)$ are independent of $n_i$,
\item $\psi_i(t)$ is a Pisot polynomial,
\end{enumerate}
Furthermore, the largest real root of $\chi(t)$ approaches a Pisot number as $n_i\to \infty$.  
\end{prop}

\begin{proof}
Without loss of generality, we may assume that $i=1$. 
The terms in  (\ref{E:charpoly}) has $t^{n_1}$ if and only if $1 \not\in I$. Thus, we have:
\[\begin{aligned} \chi(t) \ =\ & \sum_{1 \not\in I}  (-1)^{|I|-1}\, t^{\sum_{i \not\in I} n_i}\, \det M_{t,I} + \sum_{1 \in I}  (-1)^{|I|-1}\, t^{\sum_{i \not\in I} n_i}\, \det M_{t,I}\\\ =\ &  t^{n_1} \sum_{1 \not\in I}  (-1)^{|I|-1}\, t^{N-n_I -n_1}\, \det M_{t,I} + \sum_{1 \in I}  (-1)^{|I|-1}\, t^{N-n_I }\, \det M_{t,I} \end{aligned} \]
where $N=\sum_{i=1, k} n_i$ and $n_I = \sum_{i \in I} n_i$.

Note that if $1 \not\in I$, then $N-n_I -n_1$ is the sum of $n_i$'s for $i\in \{2, \dots, k \} \setminus I$, and if $1 \in I$, then $N-n_I$ is the sum of $n_i$'s for $i\in \{2, \dots, k \} \setminus I$.
By setting \[ \psi_1(t) = \sum_{1 \not\in I}  (-1)^{|I|-1}\, t^{N-n_I -n_1}\, \det M_{t,I} \] and \[\phi_1(t)=\sum_{1 \in I}  (-1)^{|I|-1}\, t^{N-n_I }\, \det M_{t,I},\] we get the first condition. 
From (\ref{E:charpoly}), we see that $\psi_1(t)$ is, in fact $\chi_{M_{\{2, \dots, k\}}, (n_2, \dots, n_k)}$. Let $\lambda_{n_1}$ be the largest real root of $\chi(t)$ and $\lambda'$ be the largest real root of $\psi_1(t)$.
Then, we have: \[ \lambda' = \lim_{n_1 \to \infty} \lambda_{n_1}.\]
It follows that $\lambda'$ is in the Dynamical Spectrum. 

Suppose $\lambda'$ is a Salem number. Since $\chi(t)$ is a reciprocal polynomial for all $n_1 \in \mathbb{N}$, we have $\psi_1(t)=\pm \phi_1(t)$. Using the cofactor expansion as in the proof of Lemma \ref{L:charpoly}, we see that 
\[\begin{aligned} \det (\omega-t I)&\  =\  (-t)^{n_1} \psi_1(t)\\& + (-1)^{n_1-1} \det \left[\begin{array}{c|c|ccc|ccc|ccc} d-t &d_1&0&\cdots & d_2& & \cdots& & & & d_k\\ \hline -c_1 &-m_{1,1}&0&\cdots & -m_{1,2}& & & & & & -m_{1,k}\\ \hline -c_2 &-m_{2,1}&-t& & & & & & & & -m_{2,k}\\  &&1&\ddots & & & & & & & \\   &&&1 & -t& & & & & &\\ \hline  &&& & & & & & & &\\ \vdots&\vdots&& & & &\ddots & & & &\\ &&& & & & & & & &\\ \hline -c_k &-m_{k,1}&& & & & & &-t & & -m_{k,k}\\&&& & & & & &1 &\ddots & \\&&& & & & & & &1 &-t \end{array}\right] .\end{aligned}\] 
Note that $\phi_1(t)$ is the second term in the above equation.
To compute the determinant in the second term, we use a cofactor expansion along the second column. Then, the second term is given by 
\[ (-1)^{n_1-1} ( d_1 A(t) + m_{1,1} \psi(t) + \sum_{j=2}^k m_{j,1} B_j(t))\]
where $A(t)$ is a polynomial of degree $d_A=n_2+ n_3+ \dots+ n_k$ with the leading coefficient $\pm c_1$, and $B_j(t)$ is a polynomial of degree strictly less than the degree of $A(t)$. Since the degree of $\psi_1(t)$ is $1+n_2+ n_3+ \dots+ n_k$, to have $\psi_1(t)=\pm \phi_1(t)$, we need: 
\[ d_1 A(t) + \sum_{j=2}^k m_{j,1} B_j(t)= d_1 c_1 t^{d_A} + \text{ lower order terms} =0 \qquad \forall t.\] It follows that $d_1 c_1 =0$. This is a contradiction because $M \in \mathcal{M}_d$ implies $c_i, d_i \ne 0$ for all $j$ by Proposition \ref{P:nozerosw}.
\end{proof}
\begin{thm}\label{P:mono}
Suppose $\omega$ and $\omega'$ have distinct orbit data $\mathcal{O}(\omega)=(M, (n_1, \dots, n_k))$ and $\mathcal{O}(\omega')=(M, (n'_1, \dots, n'_k))$ respectively. If $n_i \le n'_i$ for all $i=1, \dots, k$, $(n_1, \dots, n_k) \ne (n_1', \dots, n_k')$, and $\lambda(\omega)>1$, then $\lambda(\omega) \lneq \lambda(\omega')$. 
\end{thm}

\begin{proof}
It is sufficient to consider $n_1 < n_1'$ and $n_i= n_i'$ for all $i \ne 1$.
For $\omega$ with $\mathcal{O}(\omega)=(M, (n_1, \dots, n_k))$, ss we have seen in Proposition \ref{P:n1}, the spectral radius of $\omega$ is the largest real root of \[\chi(t) = t^{n_1} \psi(t) + \phi(t), \quad \text{where}\]
\[\psi(t) \ =\   \sum_{1 \not\in I}  (-1)^{|I|-1}\, t^{N-n_I -n_1}\, \det M_{t,I}  , \quad \phi(t) \ =\  \sum_{1 \in I}  (-1)^{|I|-1}\, t^{N-n_I }\, \det M_{t,I}.\]
Consider the system of equations 
\begin{equation}\label{E:system} S(n) : \left \{ \begin{aligned} &y = - \phi(t)/\psi(t),\\& y= t^{n}\\ \end{aligned} \right. . \end{equation}
The spectral radius of $\omega$ is given by the $t$-coordinate of the unique intersection greater than $1$ of $S(n_1)$ defined in (\ref{E:system}).
Notice that both $\phi(t)$ and $\psi(t)$ are independent of $n_1$. By Proposition \ref{P:n1}, we see that $\psi(t)=0$ has a unique real root $t_\star>1$ and $\phi(t)=0$ has no real root $>1$. Thus, we have: \[ \lim_{t \nearrow t_\star} -\frac{\phi(t)}{\psi(t)} = +\infty. \]
Let $\lambda_n$ be the unique solution of $S(n)$ in $\{ t>1\}$.
Since $t^{n_1'} > t^{n_1}$ for all $t>1$ and $\lambda_n < t_\star$, if $\lambda_{n_1'}< \lambda_{n_1}$ then $y = - \phi(t)/\psi(t)$ must intersect $t^{n_1'}$ for some $t \in ( \lambda_{n_1}, t_\star)$. In other words, the system $S(n_1')$ has more than one solution. It follows that the $t$-coordinate of the unique solution $>1$ in  (\ref{E:system}) increases with $n_1$.
\end{proof}

\begin{proof}[Proof of Theorem B]
Suppose $f$ is a birational map on $\mathbf{P}^2(\mathbb{C})$, and $F$ is an algebraically stable modification of $f$ on a rational surface $X$. 
Note that the orbit data of $f$ are determined by the induced action, $F_*$, on $\text{Pic}(X)$. 
If $F$ is an automorphism, we have $F_* \in W_\infty$. Thus, combining Theorem \ref{P:mono} and Proposition \ref{P:pisot}, we have the desired conclusion. 
\end{proof}

\subsection{Limits of Salem numbers in $\Lambda_S$}\label{SS:limit}
Suppose $I \subset \{1, \dots, k\}$ and $n_\star$ is a positive integer. Let $\bar n= (n_1, \dots, n_k)$ be a $k$-tuple with $n_i = n_\star$ for $i \in I$. 
Suppose $\chi(t)$ is the characteristic polynomial (up to a product of cyclotomic polynomials) of an element with orbit data $(M=\begin{bmatrix} d&\bar d\\-\bar c&-m \end{bmatrix}, \bar n=(n_1,n_2, \dots, n_k))$ defined in (\ref{E:charpoly}), and let $\lambda_{M,\bar n}$ be the largest real root of $\chi(t)$. 

Recall that $M_I$ is the principal submatrix of $M$ obtained by deleting the $i+1^\text{th}$ column and row for all $i \not\in I$.
By sorting subsets of $\{1, \dots, k\}$ based on the size of their intersection with $I$,  we see that the characteristic polynomial defined in (\ref{E:charpoly}) becomes:
\[ \chi(t) = \sum_{j=0}^{|I|} t^{(|I|-j) \cdot n_\star } \psi_j (t), \] where
\begin{equation}\label{E:psi}
 \psi_j(t) \ =\  \sum_{|I \cap J|=j }(-1)^{|J|-1} t^{N-n_J-(|I|-j)  \cdot n_\star} \det M_{t,J}, \quad N=n_1+n_2+ \cdots +n_k, \ \ n_J = \sum_{\ell \in J} n_\ell. 
 \end{equation}
 Since $|I \cap J|= j$, a index set $J$ has exactly $j$ $n_\star$. Thus, no term in $N-n_J-(|I|-j)  \cdot n_\star$ is equal to $n_\star$.
Also, let us set $\chi_I (t)$ be the polynomial corresponding to $(M_I, \tilde n)$ where $\tilde n$ is $(k-|I|)$-tuple consists of $n_i, i \not\in I$ defined in (\ref{E:charpoly}). Then, we see that $\psi_{0}(t) =\pm \chi_I(t)$. For example, suppose $\bar n=(n_1, n_2, n_\star, n_\star, n_\star)$. Then $I=\{3,4,5\}$, $\tilde n=(n_1, n_2)$, and we have:
\[ \begin{aligned} \psi_0(t) \ =\ & (-1)^{-1} t^{n_1+n_2} \det M_{t, \emptyset} \\&+ (-1)^0 t^{n_1} \det M_{t,\{n_2\}}+ (-1)^0 t^{n_2} \det M_{t,\{n_1\}}\\&+ (-1)^1  \det M_{t,\{n_1,n_2\}}.\\ \end{aligned}\]
Furthermore, since $\chi(t)$ is a self-reciprocal polynomial for all $n_\star \in \mathbb{N}$ and all polynomials $\psi_{j}(t)$, $j=0, \dots, |I|$ are independent of $n_\star$, we see that: 
\[ \psi_{|I|-j }(t) = \pm t^{\deg \psi_j(t)} \psi_j(1/t),\quad j=0, \dots, |I|. \]
By letting $n_\star \to \infty$, we get:
\[ \lim_{n_\star \to \infty} \lambda_{M, \bar n} = \lambda_P,\]
where $\lambda_P$ is the largest root of $\psi_{0}(t)$. By Blanc and Cantat \cite{blancdynamical}, we see that $\lambda_P$ is the unique largest real root $>1$ of $\psi_{0}(t)$.

\begin{prop}\label{P:pisot}
Suppose $\chi(t)$ is the characteristic polynomial (up to a product of cyclotomic polynomials) of an element with orbit data $(M \in \cup \mathcal{M}_d, (n_1,n_2, \dots, n_k))$ defined in (\ref{E:charpoly}). Let $\lambda_{(n_1, n_2, \dots, n_k)}$ be the largest real root of $\chi(t)$. 
Fix a proper nonempty subset $I \subset \{1, \dots, k\}$ and let the polynomial $\psi_{0}(t)$ defined in (\ref{E:psi}). Then we have 
\[ \lim_{n_i \to \infty, i \in I} \lambda_{(n_1, n_2, \dots, n_k)} = \lambda_P,\]
where $\lambda_P$ is the largest real root of $\psi_{0}(t)$. In fact, $\lambda_P$ is a Pisot number, and $\lambda_{(n_1, n_2, \dots, n_k)}$ approaches to $\lambda_P$ from below.
\end{prop}

\begin{proof}
We may assume that $n_i$ are all equal for $i\in I$. From the discussion above, it is clear that the limit is a root $\lambda_P$ of $\psi_{0}(t)$.  $\lambda_P$ is either a Salem number or a Pisot number. It follows that $\psi_{0}(t)$ has a unique real root $>1$. With the essentially same argument as in Proposition \ref{P:n1}, we see that $\lambda_P$ is a Pisot number.
\end{proof}


\section{Salem and Pisot numbers in the Weyl Spectrum}\label{S:sp}
In this section, we inductively determine all possible Salem and Pisot numbers in the Weyl spectrum. By Theorem \ref{T:charpoly}, we need to find all possible orbit data $\mathcal{O}(\omega)$ for $\omega \in \cup_d \Omega_d$ to determine all possible Salem numbers and reciprocal quadratic integers in $\Lambda_W =\Lambda_\mathbb{C}$.

\begin{lem}\label{L:alllength}
Let $M \in  \mathcal{M}_d$, where $d\ge 2$, be a $(k+1)\times (k+1)$ matrix. Then, for all $k$-tuple $\bar n$ of positive integers, there is $\omega \in \Omega_d$ such that $\mathcal{O}(\omega) =(M, \bar n)$.
\end{lem}

\begin{proof}
For any positive integer $m$, let $s = s_1 s_2 \dots s_{m-1}$ be a product of $m-1$ reflections. It follows that $s \in \cup_n W_n$ is permuting $e_i$ for $i=1, \dots, m$ cyclically: \[ s\ :\ \left\{ \  \begin{aligned} &\ e_1 \mapsto e_2 \mapsto \cdots \mapsto e_m \mapsto e_1\\&\ e_j \mapsto e_j \qquad \text{for all }\ j \ge +1 \end{aligned} \right. .\] Thus, using the formula in (\ref{E:ws}), we can construct any orbit data with a given orbit permutation $M$. 
\end{proof}

\begin{lem}\label{L:base1}
Let $M \in \mathcal{M}_d$, where $d\ge 2$, be a $(k+1)\times (k+1)$ matrix. If $M'$ is a matrix obtained by permuting the columns of $M$ except the first column, then $M' \in \mathcal{M}_d$. Similarly, if $M'$ is a matrix obtained by permuting the rows of $M$ except the first row, then $M' \in \mathcal{M}_d$. 
\end{lem}

\begin{proof}
If $M'$ is obtained by permuting the second and third columns, then $M' = M \cdot s_1$ where $s_1$ is the reflection through the vector $e_1- e_2$. Also, $s_1 \cdot M$ results in a matrix obtained by permuting the second and third rows of $M$. This lemma follows by applying this argument repeatedly.
\end{proof}

Also, we have:
 
\begin{lem}\label{L:base2}
Let $M\in \mathcal{M}_d$ be a $(k+1)\times (k+1)$ matrix.  Suppose $M'$ is a matrix obtained by permuting the $2^\text{nd}$ to $(k+1)^\text{th}$ rows of $M$. Then by renaming $e_1, \dots, e_k$ if necessary, $M'$ is equal to a matrix obtained by permuting the $2^\text{nd}$ to $(k+1)^\text{th}$ columns of $M$.
\end{lem} 

\begin{proof}
Suppose $M' \in \mathcal{M}_d$ is obtained by permuting the second and third columns of $M \in \mathcal{M}_d$. Then we have 
$ M' = M \cdot s_1$ and therefore: \[ s_1 \cdot M' \cdot s_1 = s_1 \cdot M  \in \mathcal{M}_d.\]
In all other cases, we can achieve the desired result by applying the above argument repeatedly. 
\end{proof}

Suppose $M \in \mathcal{M}_d$ is a $(k+1)\times (k+1)$ matrix. Let $S_M \subset W_k$ be the set of all permutations of $e_1, \dots e_k$:
\[ S_M := \{ w \in W_k: w(e_0) = e_0 \}. \]
Then, by above Lemmas \ref{L:base1}- \ref{L:base2}, we see that the orbit of $M$ is in $\mathcal{M}_d$:
\[ \mathcal{O}_M:=\{ M \cdot s : s \in S_M\} \subset \mathcal{M}_d. \]
Since $\mathcal{M}_d$ is finite, we can write $\mathcal{M}_d$ as a finite disjoint union by choosing one representative for each orbit:
\[ \mathcal{M}_d = \sqcup \mathcal{O}_{M_i} \qquad \text{where} \qquad M_{i} \in \mathcal{O}_{M_i},\ \ M_{i} \not\in \mathcal{O}_{M_j} \text{ for } j \ne i.\]

For instance, we have: 
\[ \mathcal{M}_2 = \mathcal{O}_M, \qquad \text{where}\qquad  M= \begin{bmatrix} 2&1&1&1\\-1&0&-1&-1\\-1&-1&0&-1\\-1&-1&-1&0 \end{bmatrix},\]
\[ \mathcal{M}_3 = \mathcal{O}_M, \qquad \text{where}\qquad  M= \begin{bmatrix} 3&2&1&1&1&1\\-2&-1&-1&-1&-1&-1\\-1&-1&-1&0&0&0\\-1&-1&0&-1 &0&0\\-1&-1&0&0&-1&0\\ -1&-1&0&0&0&-1\end{bmatrix},\]
and 
\[ \mathcal{M}_4 = \mathcal{O}_{M_1} \sqcup \mathcal{O}_{M_2},\] where \[ M_1= \begin{bmatrix} 4&3&1&1&1&1&1&1\\-3&-2&-1&-1&-1&-1&-1&-1\\-1&-1&-1&0&0&0&0&0\\-1&-1&0&-1 &0&0&0&0\\-1&-1&0&0&-1&0&0&0\\ -1&-1&0&0&0&-1&0&0\\ -1&-1&0&0&0&0&-1&0\\ -1&-1&0&0&0&0&0&-1\end{bmatrix}, \ \ M_2= \begin{bmatrix} 4&2&2&2&1&1&1\\-2&-1&-1&-1&-1&-1&0\\-2&-1&-1&-1&-1&0&-1\\-2&-1&-1&-1 &0&-1&-1\\-1&-1&-1&0&0&0&0\\ -1&-1&0&-1&0&0&0\\ -1&0&-1&-1&0&0&0\end{bmatrix}.\]

\vspace{1ex}


%
%
%
%

\begin{lem}\label{L:smallerdegree}
Suppose $M = \begin{bmatrix} d & \bar d \\ - \bar c & - m\end{bmatrix} \in \mathcal{M}_d$. Then there exists an $\omega \in \Omega_d$ such that \[ \deg(M \cdot \omega ) \lneq d. \]
In this case, we have: 
\begin{itemize}
\item $ M\cdot \omega (e_p) \cdot e_0 = M (e_p) \cdot e_0$ \ \ for $ p \ne 0, i,j,k$, and
\item $ M\cdot \omega (e_p) \cdot e_0 = d - d_q-d_r$ \ \ for $ \{p,q,r\} = \{i,j,k\} $.
\end{itemize}
\end{lem}

\begin{proof}
Suppose $M \in W_n$ for some $n\ge 3$. By Lemma \ref{L:facts}, there are three indices $i,j,k$ such that $d_i+d_j+d_k \ge d+1$. Let $\omega \in W_n$ be a reflection through the vector $e_0 - e_i -e_j-e_k$, that is,
\[\begin{aligned}  \omega \ : \ &e_0 \mapsto 2 e_0 - e_i-e_j-e_k, \\ &e_p \mapsto e_0 - e_i-e_j-e_k + e_p \quad \text{for } \ p = i,j,k, \\ & e_p \mapsto e_p \quad \text{otherwise}. \end{aligned} \]
Thus we have \[ \deg (M \cdot \omega) \ =\ 2d - d_i -d_j-d_k \le d-1.\] The second part of this lemma follows from simple matrix multiplication. 
\end{proof}

Let $\kappa_{i,j,k}$ be the reflection through the vector $e_0 - e_i -e_j-e_k$. Since $\kappa_{i,j,k}$ is an involution, 
we see that any $M \in \mathcal{M}_d$ can be obtained by $M' \cdot \kappa_{i,j,k}$ for some $M' \in M_{d'}$ with $d'<d$. 
In fact, by the previous Lemma \ref{L:smallerdegree}, we see that
\[ d'_i+d'_j+d'_k = 3d - 2 (d_i+d_j+d_k) \le d-2 \]
where  $d_i',d_j',d_k'$ are the corresponding entries in the first row of the matrix representation of $M' = \begin{bmatrix} d'&\bar d' \\ - \bar c'&m' \end{bmatrix} $.
Notice that $d_i'$ could be zero if $d_j+d_k = d$. 

Suppose $M =\begin{bmatrix} d&\bar d\\-\bar c&-m\end{bmatrix} \in \mathcal{M}_d$ is represented by a $(k+1)\times (k+1)$ matrix. For each $s \le d-1$,  let: 
\[ \begin{aligned} &I_{s,3}(M) = \{ (i_1, i_2,i_3) : d_{i_1}+d_{i_2}+d_{i_3} = s , 1\le i_1\le i_2\le i_3\le k\},\\ 
&I_{s,2}(M) = \{ (i_1, i_2,k+1) : d_{i_1}+d_{i_2}= s , 1\le i_1\le i_2\le k\},\\ 
&I_{s,1}(M) = \{ (i_1, k+1,k+2) : d_{i_1}= s , 1\le i_1\le k\},\\
&I_{s,0}(M) = \{ ( k+1,k+2,k+3) \}.\\ 
\end{aligned}\] 
And let: \[ I_s(M) = I_{s,3}(M) \cup  I_{s,2}(M) \cup I_{s,1}(M) \cup I_{s,0}(M).\]
Notice that this set $I_s(M)$ is not necessarily non-empty.
Also, let us define: 
\[h_s(\mathcal{M}_d) = \{ M \cdot \kappa_{i,j,k} : M=\begin{bmatrix} d&\bar d\\-\bar c&-m\end{bmatrix} \in \mathcal{M}_d, (i,j,k) \in I_s(M) \}. \]

\begin{lem}\label{L:hmap}
Suppose $0 \le s \le d-1$. Then \[ h_s(\mathcal{M}_d) \subset \mathcal{M}_{2 d-s}.\]
\end{lem}

\begin{proof}
Suppose $M=\begin{bmatrix} d&\bar d\\-\bar c&-m\end{bmatrix} \in \mathcal{M}_d$. Also, suppose $d_{i_1}+d_{i_2}+d_{i_3} = s $, where $1\le i_1\le i_2\le i_3\le k+3$. 
Then we have: 
\[ \begin{aligned} M \cdot \kappa_{i,j,k}\ :\ &e_0 \mapsto M (2e_0 - e_i-e_j-e_k), \\ & e_s \mapsto M( e_0 - e_i-e_j-e_k + e_s) \qquad \text{ for } s \in \{ i,j,k\}, \\& e_s \mapsto M(e_s) \qquad \text{ for } s \not\in \{i,j,k\}. \end{aligned}\]
Thus, we see that the first row of $M \cdot  \kappa_{i,j,k}$ consists of strictly positive numbers, and therefore $M \cdot  \kappa_{i,j,k} \in \mathcal{M}_{2d-s}$.
\end{proof}

\begin{thm}\label{T:allbase}
For each $d\ge 2$, the set $\mathcal{M}_d$ is determined recursively:
\[ \mathcal{M}_{d} =\bigcup_{i=1}^{d/2} h_{d-2i}(\mathcal{M}_{d-i}),  \]
where
\[ \mathcal{M}_2 = \left \{  \begin{bmatrix} 2&1&1&1\\-1&0&-1&-1\\-1&-1&0&-1\\-1&-1&-1&0 \end{bmatrix},\   \begin{bmatrix} 2&1&1&1\\-1&-1&0&-1\\-1&0&-1&-1\\-1&-1&-1&0 \end{bmatrix}, \  \begin{bmatrix} 2&1&1&1\\-1&-1&-1&0\\-1&0&-1&-1\\-1&-1&0&-1 \end{bmatrix} \right\}. \]
\end{thm}

\begin{proof}
This follows from Lemmas \ref{L:smallerdegree} and \ref{L:hmap}.
\end{proof}

\begin{thm}\label{T:allsalem}
The set $\Lambda_{S,d}$ of spectral radii of $\omega \in \Omega_d$ is given by 
\[ \Lambda_{S,d} = \{ \rho(\chi_{M, \bar n} ):  \bar n = (n_1, \dots, n_k), n_i \ge 1, k+1 = |M|, M \in \mathcal{M}_d\},\]
where $\rho$ denotes the largest real root, $|M|$ is the number of rows of $M$, and $\chi_{M, \bar n}$ is the polynomial defined in (\ref{E:charpoly}).
Thus, the set $\Lambda_S$ of all Salem numbers and reciprocal quadratic integers in the Weyl spectrum is given by: 
\[ \Lambda_S = \cup_d \Lambda_{S,d}. \]
\end{thm}

\begin{proof}
The first part of this theorem is an immediate consequence of Theorem \ref{T:charpoly} and Lemma \ref{L:alllength}.
Also, we have: \[ \mathcal{M}_2 = \mathcal{O}_M, \qquad \text{where}\qquad  M= \begin{bmatrix} 2&1&1&1\\-1&-1&0&0\\-1&0&-1&0\\-1&0&0&-1 \end{bmatrix}.\] Thus, we can determine all matrices in $\mathcal{M}_d$
by Theorem \ref{T:allbase}. Consequently, we have the set of all Salem numbers and reciprocal quadratic integers in the Weyl Spectrum. 
\end{proof}

Suppose $M \in \mathcal{M}_d$ is given by $(k+1) \times (k+1)$ matrix. Recall that for a proper subset $I \subset \{ 1, 2, \dots, k\}$,  $M_I$ is the principal sub-matrix of $M$ obtained by deleting the $(i+1)^\text{th}$ column and row for $i \not\in I $.
And let $\mathcal{M}'_d$ be the set of all such sub-matrices $M_I$, that is,
\begin{equation}\label{E:setM} \mathcal{M}'_d = \{ M_I : I \subsetneq \{ 1, \dots, k\}, k+1=|M|, M \in \mathcal{M}_d\}.\end{equation}

\begin{thm}\label{T:allpisot}
For each $d\ge 2$, let $\Lambda_{P,d}$ be the set defined by
 \[ \Lambda_{P,d} = \{ \rho(\chi_{M', \bar n'}): \bar n' = (n_1, \dots, n_k), n_i \ge 1, k+1 = |M'|, M' \in \mathcal{M}'_d\},\]
where $\rho$ denotes the largest real root, $|M|$ is the number of rows of $M$, and $\chi_{M, \bar n}$ is the polynomial defined in (\ref{E:charpoly}).
The set $\Lambda_P$ of all Pisot numbers in the Weyl spectrum is given by the union of $\Lambda_{P,d}$ over all $d\ge 2$:
\[ \Lambda_P = \cup_d \Lambda_{P,d}. \]
\end{thm}

\begin{proof}
By Theorem \ref{P:mono} and Proposition \ref{P:pisot}, we see that all Pisot numbers can be obtained using certain sub-matrices of $M$ for some $M \in \cup_d \mathcal{M}_d$. Given that $\mathcal{M}_d$ is defined for all $d\ge 2$, we can determine all Pisot numbers in the Weyl spectrum. 
\end{proof}

\begin{proof}[Proof of Theorem A and Proof of Main Theorem]
By setting  $\Lambda_d = \Lambda_{S,d} \cup \Lambda_{P,d}$, we immediately obtain Theorem A and the Main Theorem as direct consequences of Theorems \ref{T:allsalem} and \ref{T:allpisot}. 
\end{proof}

\section{Realization of small Pisot numbers in the Dynamical Spectrum}\label{S:pisot}
A Pisot number is an algebraic integer greater than $1$ whose Galois conjugates all have a modulus less than $1$. The smallest Pisot number, $p_1 \approx 1.3247$, is the largest real root of the polynomial $t^3-t-1$.The smallest accumulation point of Pisot numbers is the golden ratio $\phi=(1+\sqrt{5})/2 \approx 1.6180$, which is the largest real root of $t^2-t-1$. Therefore, the golden ratio $\phi$ is also a Pisot number. 

All Pisot numbers smaller than the golden ratio are listed in \cite{Pisot:1955}.
Let us define the following functions:
\[ \begin{aligned}&\theta_1(n) \ : =\  t^n(t^2-t-1)+t^2-1,\\&\theta_2(n) \ := \ \frac{t^{2n} (t^2-t-1)+1}{t-1},\\&\theta_3(n) \ :=\ \frac{t^{2n+1} (t^2-t-1)+1}{t^2-1},\\&\theta_s\ :=\ t^6-2 t^5+t^4-t^2+t-1.\end{aligned}\]
Then, the $n^\text{th}$ smallest Pisot number, $p_n$, is the largest real root of $P_n(t)$ where: 
\begin{equation}\label{E:pisotList} 
\begin{aligned}P_1 = \theta_1(1), \ &P_2 =\theta_1(2), \ P_3=\theta_1(3), \ P_4=\theta_3(1), \\ &P_5= \theta_1(4), \ P_6=\theta_2(2),\ P_7=\theta_1(5), \ P_8= \theta_s, \ P_9 = \theta_3(2). \end{aligned}\end{equation}
For $P_m$ where $m>9$, we have:
\begin{equation}\label{E:pisotList2} 
P_m \ =\  \left\{ \begin{aligned}& \theta_1( m/2+1) \qquad\qquad\ \   \text{if   } m>9, m \equiv 0 (\text{mod }2), \\ & \theta_2( (m-3)/4+1) \qquad\  \text{if   } m>9, m \equiv 3 (\text{mod }4), \\  &\theta_3( (m-1)/4) \qquad\qquad \text{if   } m>9, m \equiv 1 (\text{mod }4). \end{aligned} \right. \end{equation}

\subsection{Small Pisot numbers in the Weyl Spectrum}
There are three possible orbit permutations for $\omega$ with $\deg(\omega) =2$. We have:
\[ \mathcal{M}_2 = \left \{   M_1=\begin{bmatrix} 2&1&1&1\\-1&0&-1&-1\\-1&-1&0&-1\\-1&-1&-1&0 \end{bmatrix},\   M_2=\begin{bmatrix} 2&1&1&1\\-1&-1&0&-1\\-1&0&-1&-1\\-1&-1&-1&0 \end{bmatrix}, \  M_3=\begin{bmatrix} 2&1&1&1\\-1&-1&-1&0\\-1&0&-1&-1\\-1&-1&0&-1 \end{bmatrix} \right\}. \]
Let us define:
\[ M_2' =\begin{bmatrix} 2&1&1\\-1&-1&0\\-1&0&-1 \end{bmatrix}, \quad M'_3=\begin{bmatrix} 2&1&1\\-1&-1&-1\\-1&0&-1\end{bmatrix}, \quad \text{ and } \quad M'=\begin{bmatrix} 2&1\\-1&-1 \end{bmatrix}.\]
Also, let \[ f_i(n_1, n_2) \ =\   \chi_{M'_i, (n_1, n_2)} \quad \text{for} \quad i=2,3, \qquad \text{and} \] \[ f(n_1) \ = \ \chi_{M', (n_1)}, \] where $ \chi_{M, \bar n}$ is the polynomial defined in (\ref{E:charpoly}).

Since $M_2'$ and $M_3'$ are submatrices of $M_2$ and $M_3$, respectively, obtained by deleting the $4^\text{th}$ column and row, by Theorem \ref{T:allpisot}, we know that the largest real root of $f_i(n_1, n_2)$ is a Pisot number for all $i=1,2$ and $n_1, n_2 \ge 1$. Similarly, since $M'$ is a principal sub-matrix of both $M_2$ and $M_3$ obtained by deleting the $3^\text{rd}$ and $4^\text{th}$ columns and rows, the largest real root of $f(n_1)$ is a Pisot number for all $n_1 \ge 1$. 

Using the formula in (\ref{E:charpoly}), we see that:
\[ f_3(n_1, n_2) = t^{n_1+n_2+1}- 2 t^{n_1+n_2} + t^{n_1+1}+t^{n_2+1} -t^{n_1}-t^{n_2} + t-1,\]
\[ f_2(n_1, n_2) = t^{n_1+n_2+1}-2 t^{n_1+n_2} + t^{n_1+1}+t^{n_2+1} -t^{n_1}-t^{n_2} + t,\]
and
\[ f(n_1) = t^{n_1+1} -2 t^{n_1} +t-1.\]
Thus, we have the following lemmas:
\begin{lem}\label{L:pW1}
For each $n \ge 1$, we have \[ f_3(1,n) = \theta_1(n), \]
and \[ f_3(2,3) = \theta_s.\]
\end{lem}

\begin{lem}\label{L:pW2}
For each $ n \ge 1$, we have 
\[ f_2(1, 2n+2) = t^2 (t-1)\theta_2(n), \qquad\text{and} \qquad  f_2(1, 2n+3) = t^2 (t^2-1) \theta_3(n). \]
\end{lem}

\begin{lem}\label{L:pW3}
For $n_1=1$, we have \[ f(1) = t^2-t-1.\]
\end{lem}

Therefore, we have:

\begin{prop}\label{P:pisotinW}
Every Pisot number smaller than or equal to the golden ratio is realized as a dynamical degree of a birational map. 
\end{prop}

\begin{proof}
Since the dynamical spectrum is identical to the Weyl spectrum, this Proposition follows from the previous Lemmas \ref{L:pW1} --- \ref{L:pW3}.
\end{proof}

\begin{rem}
For any given $d\ge 2$, the set of orbit permutations $\mathcal{M}_d$ is finite. Given an orbit permutation $M \in \mathcal{M}_d$, Proposition \ref{P:mono} shows that the spectral radius increases with orbit lengths. Thus, it is feasible to check small Pisot numbers in $\Lambda_d$ for a given $d\ge 2$. Since the conjugacy class of an element $\omega \in W_\infty$ is infinite, we would have $\Lambda_d \subset \Lambda_d'$ for $ d\le d'$ where $\Lambda_d$ is the set of spectral radii of $\omega \in \Omega_d$ of degree $d$. Through computer experiments, we have made the following interesting observations:
\begin{itemize}
\item In $\Lambda_2$, the smallest Pisot number $p_\star$ strictly bigger than the golden ratio is a root of 
\[ t^7-2 t^6+t^5-t^4+t^3-t^2+t-1=0,\] and it is approximately $ 1.64073$. This Pisot number is the limit of spectral radii of $\omega_n \in \Omega_2$ as $n\to \infty$, where \[ \mathcal{O}(\omega_n) = (M_3, \bar n =(2,4,n)).\]
Additionally, this Pisot number can be realized as a dynamical degree of a birational map of degree 2. 
\item In $\Lambda_3$, the smallest Pisot number strictly bigger than the golden ratio is again $p_\star$. This Pisot number  is the limit of spectral radii of $\omega_n \in \Omega_3$ as $n\to \infty$ where \[ \mathcal{O}(\omega_n) = \left(\begin{bmatrix}3&1&2&1&1&1\\-2&-1&-1&-1&-1&-1\\-1&-1&-1&0&0&0\\-1&0&-1&0&0&-1\\-1&0&-1&-1&0&0\\-1&0&-1&0&-1&0\end{bmatrix}, \bar n =(1,2,1,3,n)\right).\]
\item In $\Lambda_3\setminus \Lambda_2$, the smallest Pisot number strictly bigger than the golden ratio is the root of \[ t^7-t^6-2 t^5+2 t^3+t^2-t-1=0\] and it is approximately $ 1.75827$.
\item Based on computer experiments, it seems that $p_\star$ is the smallest Pisot number strictly greater than the golden ratio in the Weyl spectrum. However, we have not ruled out the possibility of the existence of a Pisot number between the golden mean and $p_\star$ in $\Lambda_d$ for sufficiently large $d$. 
\end{itemize}
\end{rem}
\subsection{Realization}\label{SS:realization}
Notice that for each $i=1,2$ and $n_1, n_2 \ge 1$, we have:
\[ \rho(f_i(n_1, n_2) ) = \lim_{n_3 \to \infty} \rho(\chi_{M_i, (n_1, n_2, n_3)}),\] where $\rho$ denotes the largest real root. 
Diller \cite{Diller:2011} provided a method for constructing a quadratic birational map properly fixing a cubic curve. 

\begin{thm}\cite[Theorem~1]{Diller:2011}\label{T:diller}
Let $C\subset \mathbf{P}^2(\mathbb{C})$ be an irreducible cubic curve. Suppose $p_1, p_2,p_3 \in C_\text{reg}$ are points in the smooth locus of $C$, $b \ne 0 \in C_\text{reg}$, and $a \in \mathbb{C}^*$. Then there exists the unique quadratic birational map $f$ properly fixing $C$ with $\mathcal{I}(f) = \{p_1,p_2,p_3\}$ and the restriction $f|_{C_\text{reg}} : z \to a z+b$ if and only if $a(p_1+p_2+p_3) = 3 b$.
In this case the points of indeterminacy for $f^{-1}$ are given by $\{ a p_i-2b, i=1,2,3\}$.
\end{thm}

Let us summarize a method from \cite{Diller:2011}. Set $C=\{\gamma(t) = [1:t:t^3]: t \in \mathbb{C} \cup \{\infty\} \} \subset \mathbf{P}^2(\mathbb{C})$. Let $p_1= \gamma(t_1), p_2=\gamma(t_2), p_3=\gamma(t_3),a \in \mathbb{C}^*$, and set $b=1-a$. Using the same notation as in Example \ref{eg:bp2}, Diller \cite{Diller:2011} shows that for sufficiently large $n_3$, there exist $f\in \bptwo$ with $\mathcal{O}(f) = (M_3, (n_1,n_2,n_3))$ determined by $t_1,t_2, t_3$ and $a$ satisfying:
\begin{equation}\label{E:cyclic}\begin{aligned} & (t_2-1)/a^{n_1} +1+t_2+t_3=0,\\& (t_3-1)/a^{n_2} +1+t_1+t_3=0,\\& (t_1-1)/a^{n_3} +1+t_2+t_1=0,\\ &t_1+t_2+t_3= 3(1-a)/a. \end{aligned}\end{equation}
The first three conditions in (\ref{E:cyclic}) result from the fact that the sum of parameters at all intersection points of a line with the cubic curve $C$ must be equal to zero. The last equation is due to Theorem \ref{T:diller}.

To obtain a birational map $f\in \bptwo$ with orbit data $(M_3' ,(n_1,n_2))$, the parameters $t_1,t_2, t_3$ and $a$ need to satisfy the first, second and last equations in (\ref{E:cyclic}) while failing the third equation for all $n_3\in \mathbb{Z}_{>0}$. By counting the number of equations, we see that there are infinitely many such $t_1,t_2, t_3$ and $a$. In fact, if $f$ is birationally equivalent to an automorphism, then $a$ must be equal to $\lambda(f)$ or a Galois conjugate of the dynamical degree $\lambda(f)$. Thus, by choosing $a \in \mathbb{Z}$, we can construct a birational map with orbit data $(M_3' ,(n_1,n_2))$ that properly fixes a cubic curve $C$. A detailed discussion of constructing a birational map with given parameter values $t_1,t_2,t_3$, and $a$ is available in \cite{Diller:2011, Bedford-Diller-K, Kim:2022}. As a summary, we have:

\begin{proof}[Proof of Theorem C] Theorem C follows from Proposition \ref{P:pisotinW}. Let us describe a concrete construction using the method in \cite{Diller:2011}. 

For $(n_1, n_2) \in \{ (2,3), (1, n), n\ge 1\}$, let $t_1, t_2, t_3$ and $a$ be complex numbers satisfying:
\begin{equation}\label{E:f3}
\begin{aligned} &\frac{1}{a^{n_1}} (t_2-1)+1+ t_2+t_3 =0,\\ &\frac{1}{a^{n_2}} (t_3-1)+1+ t_1+t_3 =0,\\ &t_1+t_2+t_3 = 3 \frac{1-a}{a}, \\ & \frac{1}{a^{k}} (t_1-1)+1+ t_2+t_1 \ne 0 \quad \text{ for all } k \ge 1.
\end{aligned}
\end{equation}
The first three equations in (\ref{E:f3}) can be written in matrix form as:
\[ \begin{bmatrix} 0& 1/ a^{n_1} +1&1\\ 1&0& 1/a^{n_2}+1\\1&1&1 \end{bmatrix} \cdot \begin{bmatrix} t_1\\t_2\\t_3 \end{bmatrix} \ = \ \begin{bmatrix} 1/a^{n_1} -1\\1/a^{n_2}-1 \\ 3(1-a)/a \end{bmatrix}. \]
Since the determinant of the $3 \times 3$ matrix on the left-hand side is $1/a^{n_1+n_2} + 1/ a^{n_2}+1$, for any choice of $a>1$, there is a non-trivial solution. Since the last condition in (\ref{E:f3}) is an open condition, we see that there are always such  $t_1, t_2, t_3$ and $a$. Then, by Diller \cite{Diller:2011}, there is a birational map with orbit data $(M_3', (n_1, n_2))$. 

Similarly, there are complex numbers $t_1, t_2, t_3$ and $a$ satisfying:
\begin{equation}\label{E:f2}
\begin{aligned} &\frac{1}{a^{n_1}} (t_2-1)+1+ t_2+t_3 =0,\\ &\frac{1}{a^{n_2}} (t_1-1)+1+ t_1+t_3 =0,\\ &t_1+t_2+t_3 = 3 \frac{1-a}{a}, \\ & \frac{1}{a^{k}} (t_3-1)+1+ t_2+t_1 \ne 0 \quad \text{ for all } k \ge 1.
\end{aligned}
\end{equation}
Thus, a birational map with orbit data $(M_2', (1, n))$ exists.

Finally, we see that there exists a birational map with orbit data $(M',(1))$ using the conditions:
\begin{equation}\label{E:f}
\begin{aligned} &\frac{1}{a^{n_1}} (t_2-1)+1+ t_2+t_3 =0,\\  &t_1+t_2+t_3 = 3 \frac{1-a}{a}, \\ &\frac{1}{a^{\ell }} (t_1-1)+1+ t_1+t_3 \ne 0 \quad \text{ for all } \ell \ge 1,\\& \frac{1}{a^{k}} (t_3-1)+1+ t_2+t_1 \ne 0 \quad \text{ for all } k \ge 1.
\end{aligned}
\end{equation}
By Lemmas  \ref{L:pW1} --- \ref{L:pW3}, we see that the set of dynamical degrees of the above birational maps is the set of all Pisot numbers smaller than or equal to the golden ratio. 
\end{proof}

\section{Limit points in Pisot numbers in Weyl spectrum. }\label{ApendA}
 Due to Amara \cite[Theorem~3.1]{Amara:1966}, all limit points of Pisot numbers less than $2$ are given by
\[ \alpha_1 =\beta_1<\alpha_2<\beta_2<\alpha_3<\hat \theta'_2 < \beta_3<\alpha_4 <\cdots <\alpha_n<\beta_n <\cdots <2 \] where $\alpha_n$ is the largest real root of \[t^{n+1}-2 t^n+t-1=0,\] $\beta_n$ is the largest real root of \[t^{n+1} - (t^{n+1}-1)/(t-1)=0,\] and $\hat\theta'_2$ is the largest real root of \[t^4-t^3-2 t^2 +1=0.\]
Approximate values for the seven smallest limit points of Pisot numbers are as follows: \[ \alpha_1 \approx 1.6180,\ \  \alpha_2 \approx 1.7549,\ \   \beta_2 \approx 1.8393,\ \   \alpha_3 \approx 1.8668,\] \[ \hat\theta'_2 \approx 1.9052, \ \  \beta_3 \approx 1.9276,\ \  \text{and}\ \   \alpha_4 \approx 1.9332.\]
These algebraic numbers --- $\hat\theta'_2$, $\alpha_n, \beta_n$ for $n\ge 1$ --- are Pisot numbers that approach $2$ from below. In the following subsections, we demonstrate that these Pisot numbers belong to the Weyl spectrum. Additionally, we obtain sequences of Pisot numbers in the Weyl spectrum that approach these accumulation points --- $\hat\theta'_2$, $\alpha_n, \beta_n$ for $n\ge 1$ --- from below.

\begin{thm}
All limit points of Pisot numbers less than or equal to $2$ are in the Weyl spectrum.
\end{thm}

\subsection{Pisot number $\hat\theta'_2$}
Let $\omega_{n,m} \in \Omega_3$, where
 \[ \mathcal{O}(\omega_{n,m}) = \left(\begin{bmatrix}3&1&1&2&1&1\\-2&-1&-1&-1&-1&-1\\-1&-1&0&-1&0&0\\-1&0&-1&-1&0&0\\-1&0&0&-1&-1&0\\-1&0&0&-1&0&-1\end{bmatrix}, \bar n =(1,1,n,1,m)\right).\]
 The spectral radius of $\omega_{n,m}$ is the Salem number $\lambda_{n,m}$, given by the largest real root of \[ t^{4+n+m}-t^{3+n+m}-2 t^{2+n+m}+t^{n+m}+t^{4+n}-t^{2+n}+t^n+t^{4+m}-t^{2+m}+t^m+t^4-2 t^2-t+1=0.\]
 As $m$ goes to infinity, $\lambda_{n,m}$ approaches the Pisot number $p_n$, the largest real root of \[ t^{4+n}-t^{3+n}-2 t^{2+n}+t^{n}+t^4-t^2+1=0.\]
 Thus, we have \[ \lim_{n \to \infty} p_n = \hat\theta'_2.\]
Using essentially the same argument as in Section \ref{SS:realization}, we see that the Pisot numbers $p_n$ and $\hat\theta'_2$ are realized as dynamical degrees of birational maps that are not birationally equivalent to automorphisms.

\subsection{Pisot numbers $\alpha_n, n\ge 1$}
Let $\omega_{n,m,\ell} \in \Omega_2$, where \[ \mathcal{O}(\omega_{n,m,\ell}) = \left( \begin{bmatrix} 2&1&1&1\\-1&-1&-1&0\\-1&0&-1&-1\\-1&-1&0&-1 \end{bmatrix}, \bar n=(n,m,\ell) \right).\]
The spectral radius of $\omega_{n,m,\ell}$ is the Salem number $\lambda_{n,m,\ell}$, which is the largest real root of \[ (t-1)((t^n+1) (t^m+1)(t^\ell+1)+1)-(t^{n+m+\ell}-1)=0.\] As $\ell$ goes to infinity, $\lambda_{n,m,\ell}$ approaches the Pisot number $p_{n,m}$, the largest real root of \[ t^{n+m+1}-2 t^{n+m}+t^{1+n}-t^n+t^{m+1}-t^m+t-1=0.\]
Thus, we have \[\lim_{m \to \infty}p_{n,m} = \alpha_n \qquad \text{for all } n \ge 1.\]

\subsection{Pisot numbers $\beta_n, n\ge 1$}
Let $\omega_{n,m,\ell} \in \Omega_2$, where \[ \mathcal{O}(\omega_{n,m,\ell}) = \left( \begin{bmatrix} 2&1&1&1\\-1&0&-1&-1\\-1&-1&0&-1\\-1&-1&-1&0 \end{bmatrix}, \bar n=(n,m,\ell) \right).\]
The spectral radius of $\omega_{n,m,\ell}$ is the Salem number $\lambda_{n,m,\ell}$, which is the largest real root of
 \[ t^{n+m+\ell+1} - 2 t^{n+m+\ell} + t^{n+m}+t^{n+\ell} + t^{m+\ell}- t^{1+n}-t^{1+m}-t^{1+\ell} + 2t-1=0.\] As $\ell$ goes to infinity, $\lambda_{n,m,\ell}$ approaches the Pisot number $q_{n,m}$, the largest real root of \[ t^{n+m+1}-2 t^{n+m}+t^n+t^m-t=0.\]
Since $t^{n+1}-2 t^n+1 = (t-1) (t^n - (t^n-1)/(t-1))$, we have \[\lim_{m \to \infty}q_{n+1,m} = \beta_n \qquad \text{for all } n \ge 1.\]

Again, using essentially the same argument as in Section  \ref{SS:realization}, we see that the Pisot numbers $p_{n,m}$, $q_{n,m}$, $\alpha_n$ and $\beta_n$  are realized as dynamical degrees of birational maps that are not birationally equivalent to automorphisms.

\bibliographystyle{plain}
\bibliographystyle{unsrt}
\bibliography{biblio}
\end{document}